\documentclass[twoside,11pt]{article}

\usepackage[nohyperref]{jmlr2e}
\usepackage{amsmath,amsfonts}
\usepackage[colorlinks,citecolor=blue,urlcolor=blue]{hyperref}
\usepackage{hyperref}
\usepackage{mathrsfs}
\usepackage{tcolorbox}
\usepackage{algorithmic}
\usepackage[ruled,noend]{algorithm2e}
\usepackage{xcolor}
\usepackage{framed}
\usepackage{bbm,tikz}

\usepackage[utf8x]{inputenc}
 
\newcommand{\R}{\mathbb{R}}

\newcommand{\polylog}{\mathop{polylog}}

\newtheorem{innercustomthm}{Theorem}
\newenvironment{customthm}[1]{\renewcommand\theinnercustomthm{#1}\innercustomthm}
  {\endinnercustomthm}
\newcounter{dummy}
\newtheorem{assumption}[dummy]{Assumption}

\DeclareMathOperator{\tw}{tw}

\newcommand{\CI}{\mathrel{\perp\mspace{-10mu}\perp}}

\def\IND{\mathbbm{1}}

\newcommand{\wh}{\widehat}
\newcommand{\argmin}{\mathop{\mathrm{argmin}}}
\newcommand{\X}{\mathcal{X}}

\newcommand{\EXP}{\mathbb{E}}
\newcommand{\PROB}{\mathbb{P}}

\newcommand{\var}{\mathrm{Var}}

  \newcommand{\bc}{\rm bc}
  \newcommand{\cA}{\mathcal{A}}
 \newcommand{\cB}{\mathcal{B}}
  \newcommand{\cC}{\mathcal{C}}
  \newcommand{\cD}{\mathcal{D}}
    
    \newcommand{\cS}{\mathcal{S}}
    \newcommand{\cT}{\mathcal{T}}

\firstpageno{1}

\usepackage{lastpage}
\jmlrheading{22}{2021}{1-\pageref{LastPage}}{10/20; Revised
8/21}{9/21}{20-1137}{G\'{a}bor Lugosi, Jakub Truszkowski, Vasiliki Velona, and Piotr Zwiernik}
\ShortHeadings{Structure learning by covariance queries}{Lugosi, Truszkowski, Velona, Zwiernik}

\begin{document}

\title{Learning partial correlation graphs and graphical models by covariance queries}

\author{\name G\'abor Lugosi \email gabor.lugosi@upf.edu \\
       \addr ICREA, Pg. Llu\'{\i}s Companys 23, \\
       08010 Barcelona, Spain; \\ 
       \addr Department of Economics and Business\\
       Pompeu Fabra University; and \\
       Barcelona School of Economics\\
       \AND
       \name Jakub Truszkowski \email jakub.truszkowski@bioenv.gu.se \\
       \addr LIRMM, CNRS\\ Universit{\'e} de Montpellier\\ Montpellier, France, and\\
       Department of Biological and Environmental Sciences and\\
       Gothenburg Global Biodiversity Centre \\
       University of Gothenburg \\
       Gothenburg, Sweden 
       \AND
       \name Vasiliki Velona \email vasiliki.velona@upf.edu \\
       \addr Department of Economics and Business\\
       Pompeu Fabra University; and\\
       Einstein Institute of Mathematics\\ Hebrew University of Jerusalem
\AND
       \name Piotr Zwiernik \email piotr.zwiernik@upf.edu \\
       \addr Department of Economics and Business\\
       Pompeu Fabra University; and\\
       Barcelona School of Economics\\
       Barcelona, Spain
       }

\editor{David Sontag}
\maketitle

\begin{abstract}
We study the problem of recovering the structure underlying large
  Gaussian graphical models or, more generally, partial correlation graphs.  In high-dimensional problems it is often too costly to store the entire sample covariance matrix. We propose
  a new input model in which one can query single entries of the 
  covariance matrix. We prove that it is possible to recover the support of the inverse covariance matrix with
  low query and computational complexity. Our algorithms work in a regime when this support is represented by tree-like
  graphs and, more generally, for graphs of small
  treewidth. Our results demonstrate that for large classes of graphs,
  the structure of the corresponding partial correlation graphs can be
  determined much faster than even computing the empirical covariance matrix.\end{abstract}


\section{Introduction}

Learning the graph structure underlying probabilistic graphical models
is a problem with a long history; see \cite{drton2017structure} for a recent exposition. In the
classical setting, when the number $n$ of variables is reasonably
small, this can be done by using stepwise selection procedures based
on information criteria like BIC, AIC, or using the likelihood
function; see \cite[Section 4.4]{hojsgaard2012graphical} for a
discussion.

In high-dimensional scenarios the methods proposed for Gaussian
graphical models have become particularly successful. Here the graphs are
encoded by zeros in the inverse covariance matrix (or \emph{precision matrix}) $K$. Specifically, 
an edge is present in the graph if and only if the corresponding element of $K$
is not zero and so LASSO-type learning procedures can be applied
\cite{banerjee2008model,yuan2007model}. The link between the entries
of $K$ and coefficients obtained by linearly regressing one variable
against the rest gave rise to the so-called neighbor selection
methods, see \cite{meinshausen2010stability}. In all these theoretical
developments, the sample complexity required for learning the
underlying graph is well understood. On the other hand, in these
studies either computational issues played a secondary role or the
computational budget was relatively large as all known methods require computing the sample covariance; see, for example, \cite{cai2016estimating,dasarathy2016active}. \cite{hsieh2013big} perform a careful analysis of the optimization objective used in earlier methods, which leads to a divide-and-conquer algorithm that can be applied to large data sets. More recently, \cite{zhang2018large} devised another scalable procedure based on thresholding the sample covariance matrix followed by a novel optimization step. However, the computational complexity of these approaches is still of at least quadratic order
as a function of the number of variables.

In a growing number of applications, the number of variables $n$ is so
large that a computational cost of order $n^2$ becomes
prohibitive. This means that even writing down or storing the covariance matrix 
(or an estimate of it) is not practical, rendering all aforementioned 
approaches unfeasible.
Such examples occur in some applications in biology, such as the problem of
reconstructing gene regulatory networks from large scale gene
expression data. \cite{hwang} give an extensive discussion of computational challenges of massive amounts of gene expression data and note that issues of computational complexity made researchers rely on pairwise notions of dependence; see, for example, \cite{Chan082099,zhang2011inferring}. Scalable algorithms are also of interest in phylogenetics, where the problem is to reconstruct the evolutionary relationships between tens to hundreds of thousands of DNA sequences (\cite{price2010fasttree,brown2012fast,brown2011fast}).
Another example leading to large networks is building human brain functional connectivity networks using functional
MRI data. In this setting, the data are usually aggregated to obtain a data set with a moderate number of variables that can be processed with current algorithms \cite{huang2010learning}. More efficient methods to build large networks will allow researchers to study functional MRI data at a much higher resolution.


In this paper we address the problem that quadratic complexity becomes prohibitive in modern
large-scale applications. This requires a different approach to
structure recovery, which addresses the computational issues much more
carefully. This computationally conscious approach has become more
popular in the recent years where, in selected scenarios, it was
possible to study the trade-off between statistical accuracy and
computational complexity, see
\cite{chandrasekaran2013computational,rudi2015less}. In the main part
of the paper we abstract away from statistical considerations and formulate
the following mathematical problem. Given a symmetric positive definite matrix $\Sigma$ the goal is to:
\begin{tcolorbox}
\centering
	Learn the support of $K=\Sigma^{-1}$ observing only a small fraction of $\Sigma$.
\end{tcolorbox}

 More precisely, we are interested in learning $K$ based on a small number of adaptively selected entries of $\Sigma$. 
As a main motivating example, we may think about $\Sigma$ as a covariance matrix of some underlying random vector $X=(X_1,\ldots,X_n)$, that is, the entries of $\Sigma$ are the covariances $\sigma_{ij}=\EXP[(X_i-\EXP X_i) (X_j-\EXP X_j)]$. Although formulated on the population level, our result is a prerequisite for consistently learning the graph from data based on $o(n^2)$ entries of the sample covariance matrix. Our results show that, for data sets where covariances can be estimated with sufficient accuracy, often it is possible to recover the underlying graphical model in much less time than what is required by current approaches. 

In order to formalize this approach, we propose the following input model for our analysis. The data can be accessed through queries to a \emph{covariance oracle}. The covariance oracle takes a pair of indices $i,j \in [n]$ as an input and outputs the corresponding entry 
$\sigma_{ij}$ of the matrix $\Sigma$. This is an idealized scenario that
makes the main ideas of this paper more transparent. In practice, of course, these covariances are not exactly available as
they are often estimated from data. This setup is meaningful in applications in which one may estimate, relatively easily and accurately,
the covariance between any given pair of variables. Importantly, one does not need to estimate the entire
covariance matrix. In Section \ref{sec:finite} we discuss conditions under which the idealized covariance oracle
may be replaced by a noisy version.

The \emph{query complexity} of an algorithm is the number of entries of the covariance matrix $\Sigma$
queried during the execution of the algorithm. The main findings of the paper show that, in many nontrivial cases,
the graph underlying the graphical model of $X$ can be recovered with only $\mathcal O(n \polylog(n))$ queries using
randomized algorithms. The computational complexity of the proposed algorithms is also quasi-linear. 
This is a significant decrease in complexity compared to the quadratic complexity of any recovery algorithm that uses the entire (estimated)
covariance matrix as a starting point.

Of course a so stated problem cannot be solved in full generality and the algorithms need to rely on the sparsity of $K$ induced by bounds on  related parameters of the underlying graph such as maximum degree and treewidth (to be defined below). We propose randomized procedures that recover the correct graph and have low query and computational complexity with high probability. Our main results are briefly presented in Section~\ref{sec:prelim}, after introducing some necessary definitions. The rest of the paper is devoted to a careful analysis of three main cases: trees, tree-like graphs, and graphs with small treewidth. Our main result is an algorithm for each of the three cases, which recovers the correct graph with query and computational complexity $\mathcal O(n\polylog(n))$. 

In our analysis we first assume that the true underlying graph is a tree. In this case the Chow-Liu algorithm (\cite{chow1968approximating}) is a widely used computationally efficient algorithm to search for the tree that maximizes the likelihood function. The method was originally proposed for categorical variables but it works in a much more general context with the Gaussian likelihood or any other modular criterion such as BIC, AIC as discussed by \cite{edwards2010selecting}. In our setting the Chow-Liu algorithm is equivalent to computing a maximum-weight spanning tree in the complete graph with edge weights given by the absolute values of the correlations between any two variables. Although the Chow-Liu algorithm is relatively efficient and it has good statistical properties, the computational cost is of order $\Omega(n^2)$, which may be prohibitive in large-scale applications. We introduce a simple randomized algorithm that recovers the tree with computational complexity $\mathcal O(d n\log(n))$, where $d$ is the maximum degree of the graph. For a large class of trees, when the maximum degree is small, the algorithm introduced here significantly outperforms the Chow-Liu algorithm.
At the same time, the maximum degree of a tree can be as large as $n-1$ (for the so-called star graph).
For trees with linear maximum degree, our algorithm does not improve on the Chow-Liu algorithm.
However, in view of the lower bound shown in Section~\ref{sec:101}, no algorithm can recover trees of maximum degree $d$ with less than $\Omega(dn)$ covariance queries. In this sense, our algorithm is optimal up to logarithmic factors.

More generally, in Section~\ref{sec:treelike} we study the problem of learning tree-like structures, that is, graphs whose 2-connected components have at most logarithmic size. Such graphs arise in a variety of settings; in particular, several well-studied random graph models often give rise to graphs whose 2-connected components are of logarithmic size, see~\cite{panagiotou2010maximal}. In Sections~\ref{sec:boundedTWsep} and \ref{sec:boundedTW} we study the much more general family of graphs with bounded treewidth. Bounded treewidth graphs have long been of interest in machine learning due to the low computational cost of inference in such models~\cite{chandrasekaran2012complexity,karger2001learning,kwisthout2010necessity}. Moreover, current heuristics of treewidth estimation in real-world data have indicated small treewidth in various cases of interest \cite{abu2016metric,adcock2013tree,maniu2019experimental}.

The main motivation of this work was learning Gaussian graphical models but our results hold in a much broader context. Our goal is to learn zeros in the inverse covariance matrix when covariances may be queried.
Hence, this work may be phrased as a contribution to numerical linear algebra. In the statistical setting, learning zeros in $K$ corresponds to learning the partial correlation graph. Vanishing partial correlations correspond to conditional independence (and so graphical models) in the Gaussian case but also in the non-paranormal case of \cite{liu2009nonparanormal,liu2012high}. We provide more details in Section~\ref{sec:GMPCGs}. In general, partial correlation graphs inform only about linear dependences but there are still interesting situations when much more is implied by vanishing partial correlations \cite{rossell2020dependence}.

Our work also provides a new way of performing constraint-based inference in Gaussian graphical models and partial correlation graphs. This approach was pioneered by the TETRAD program \cite{scheines1998tetrad} where vanishing tetrad constraints are used to infer the structure of hidden variable graphical models; see also \cite{MR1815675} and \cite{MR2548166}. For Gaussian graphical models, the most complete results listing the underlying model constraints were provided by \cite{sullivant2010trek} and their results are used extensively in this article. Our new paradigm of adaptively searching over the set of potential constraints may also be useful for discrete graphical models, where the underlying graphical model can be read off the partial correlation graph of appropriately extended random vector as explained by \cite{lohright}.

{The case when $\Sigma$ is observed with error leads to additional complications. Solving this problem in full generality is beyond the scope of this paper. In order to present the main ideas and some bottlenecks, in Section~\ref{sec:finite} we study the problem of recovering tree models when only a noisy covariance oracle is available.

  We finish this introduction by mentioning that, in a different context, similar algorithms to some of those introduced in this paper were proposed and analyzed for bounded-degree trees by~\cite{10.1007/978-3-642-40935-6_14} and for Bayesian networks in~\cite{NEURIPS2018_a0b45d1b}}.

\section{Preliminaries and overview of the results}\label{sec:prelim}

In this section we first present some definitions. Then we briefly overview our main results.

\subsection{Graph-theoretic definitions}

A \emph{graph} $G=(V,E)$ is a pair of finite sets $V=V(G)$ and $E=E(G)$ called \emph{vertices} and \emph{edges}, where $E$ is a set of subsets of $V$ of size two. We typically write $uv$ instead of $\{u,v\}$ to denote an edge and our graphs are simple, that is, $u\neq v$. A \emph{subgraph} of $G$ is a graph $G'=(V',E')$ such that $V'\subseteq V$ and $E'\subseteq E$. For $V'\subseteq V$, denote by $G[V']$ the graph $(V', \{uv\in E| u,v\in V'\})$, called the {\it induced subgraph} of $G$ on $V'$. If $S\subset V$ we write $G\setminus S$ to denote $G[V\setminus S]$. A path between $u$ and $v$ is a sequence of edges $v_0v_1$, $v_1v_2$,\ldots,$v_{k-1}v_k$ with $v_0=u$ and $v_k=v$. We allow for \emph{empty paths} that consist of a single vertex. Two vertices $u,v\in V$ are \emph{connected} if there is  a path between $u$ and $v$.  For $v\in V$, the set $N(v)=\{u\in V|uv\in E\}$ is the \emph{neighborhood} of $v$, $\deg(v):=|N(v)|$ is its \emph{degree}, $\Delta (G):=\max _{v\in V}\deg(v)$ denotes the \emph{maximum degree} of $G$, and $|V|$  is the \emph{size} of $G$.

A graph on $n\ge 3$ vertices is a \emph{cycle} if there is an ordering of its vertices $v_1,\dots,v_n$, such that $E=\{v_1v_{2},\ldots,v_{n-1}v_n,v_nv_1\}$. A graph is \emph{connected} if all $u,v\in V$ are \emph{connected}. A \emph{tree} is a connected graph with no cycles.

A \emph{connected component} of $G$ is a maximal, with respect to inclusion, connected subgraph of $G$. A set $S\subseteq V$ \emph{separates} $A,B\subseteq V$ \emph{in $G$} if any path from $A$ to $B$ contains a vertex in $S$. Then $S$ is called a \emph{separator of $A$ and $B$} in $G$. { When $S$ is of minimum size, it will be called a \emph{minimal separator.} Note that we allow $A$ and $B$ to intersect, in which case $A\cap B$ needs to be contained in every separator of $A$ and $B$. Denote by $\cC^S$ the set of connected components of the graph $G\setminus S$. If $S$ separates two disjoint sets $A$ and $B$, then for every $u\in A\setminus S$ and $v\in B\setminus S$, $u,v$ lie in two different connected components of $G\setminus S$.

\subsection{Partial correlation graphs and graphical models }\label{sec:GMPCGs}

Let $\Sigma=[\sigma_{ij}]$ be an $n\times n$ symmetric positive definite matrix and denote $K=\Sigma^{-1}$. For a given graph $G$ over vertex set $ [n]=\{1,\ldots,n\}$, 
denote by $\mathcal{M}(G)$ the set of  covariance matrices $\Sigma$ satisfying $K_{ij}=0$ 
for all $ij \notin E(G)$. If $\Sigma$ is a covariance matrix of a Gaussian random vector $X$ then the condition $\Sigma\in \mathcal M(G)$ can be equivalently formulated through a set of conditional independence statements because of the  equivalence (see~\cite{lauritzen1996graphical})\begin{equation}\label{gprop} K_{ij}=0\quad\Longleftrightarrow\quad X_i\CI X_j\mid X_{[n]\setminus \{i,j\}}.\end{equation}
For a given $\Sigma$, the \emph{concentration graph} $\mathcal G(\Sigma)=([n],E)$ is the graph with $E=\{ij|K_{ij}\neq 0\}$.

Given a vector $x\in \R^n$ and a subset $A\subset [n]$ denote by $x_A$ the subvector of $x$ with entries $x_i$ for $i\in A$. Similarly, for sets $A,B \subseteq [n]$ and a matrix $M\in \R^{n\times n}$, let $M_{A,B}$ denote the restriction of $M$ to rows in $A$ and columns in $B$. Write $M_A$ for $M_{A,A}$. If $\Sigma$ is the covariance of $X$ then $\Sigma_{A,B}={\rm cov}(X_A,X_B)$.  In this article we extensively use the following result of Seth Sullivant, Kelli Talaska and Jan Draisma, which translates zero restrictions on a positive definite matrix $K=\Sigma^{-1}$ in terms of minors of $\Sigma$.

\begin{theorem}\upshape{\cite[Theorem 2.15]{sullivant2010trek}}\label{th:STDrank}\itshape{
Let $G$ be a connected graph with vertex set $[n]$.
We have ${\rm rank}(\Sigma _{A,B})\leq r$ for all $\Sigma\in \mathcal{M} (G)$ if and only if there is a set $S \subseteq [n]$ with $|S |\leq  r$ such that $S$ separates $A$ and $B$ in $G$. Consequently,
${\rm rank}(\Sigma _{A,B})\leq \min\{ |S|:S \,\,\mathrm{separates}\,\, A \,\,\mathrm{and}\, \,B\}$.
Moreover, there exists a dense open subset $\Gamma$ of $\mathcal M(G)$ such that equality holds for all matrices in 
$\Gamma$. }
\end{theorem}

\noindent We call $\Gamma$ a \emph{generic} set. By a slight abuse of terminology, we call covariance matrices $\Sigma$ in $\Gamma$, as well as the corresponding random vectors $X$, generic. 

In this paper we assume that $\mathcal G(\Sigma)$ is connected or, equivalently by Theorem~\ref{th:STDrank}, that $\Sigma$ has no zero entries. 
Without this assumption the problem quickly becomes impossible to solve. For example, whether $\mathcal G(\Sigma)$ has zero or
one edge can only be decided after seeing the entire covariance matrix.
\begin{assumption}
The graph $\mathcal G(\Sigma)$ is connected.	
\end{assumption}

The following well-known characterisation of graphical models over trees will be useful. This result is well known for Gaussian tree models (see, for example, \cite{ltm}) but here we prove it formally to emphasize that this is a purely algebraic result.
\begin{lemma}\label{lem:prodrho}
If $\Sigma\in \mathcal M(T)$ for a tree $T$, 
then for every $i,j\in V$, the normalized entries $\rho_{ij}=\sigma_{ij}/\sqrt{\sigma_{ii}\sigma_{jj}}$ for $i,j\in V$ satisfy the product formula
\begin{equation}\label{eq:rhoedge}
	\rho_{ij}\;=\;\prod_{uv\in \overline{ij}}\rho_{uv},	
\end{equation}
	where $\overline{ij}$ denotes the \emph{unique} path between $i$ and $j$ in $T$. Also, if the normalized entries $\rho_{ij}$ in  $\Sigma$ satisfy (\ref{eq:rhoedge}) for some tree $T$ then $\Sigma\in\mathcal M(T)$.	
\end{lemma}
\begin{proof}
	The first implication can be easily proved by induction on the number of edges in the path $\overline{ij}$ in $T$. If $\overline{ij}$ contains a single edge then $\rho_{ij}$ trivially satisfies (\ref{eq:rhoedge}). Suppose $\overline{ij}=i-i_1-\cdots-i_k-j$. Since $\Sigma\in \mathcal M(T)$ and $i_k$ separates $\{i,i_k\}$ and $\{i_k,j\}$, Theorem~\ref{th:STDrank} assures that ${\rm rank}(\Sigma_{i i_k,i_k j})\leq 1$, or in other words, $\sigma_{ij}\sigma_{i_k i_k}=\sigma_{i i_k}\sigma_{i_k j}$. Equivalently, $\rho_{ij}=\rho_{i i_k}\rho_{i_k j}$. By the induction hypothesis, $\rho_{i i_k}$ satisfies the path-product formula over the path $\overline{i i_k}$ and so
	$$
	\rho_{ij}\;=\;\rho_{i i_k}\rho_{i_k j}\;=\;\left(\prod_{uv\in \overline{i i_k}}\rho_{uv}\right)\rho_{i_k j}\;=\;\prod_{uv\in \overline{i j}}\rho_{uv}
	$$	
	proving that (\ref{eq:rhoedge}) holds.
	
	For the other implication, suppose that the normalized entries of $\Sigma$ satisfy (\ref{eq:rhoedge}) for some tree $T$. We want to show that $\Sigma\in \mathcal M(T)$, equivalently, $K_{ij}=0$ for all $ij\notin T$. Without loss of generality, suppose $\Sigma$ is already standardized so that $\sigma_{ii}=1$ for $i\in V$. ($\mathcal M(T)$ is invariant under $\Sigma\mapsto D\Sigma D$ with $D$ diagonal). Let $k$ be any vertex that separates $i$ and $j$ in $T$. Let now $A/B$ be any partition of $V\setminus \{k\}$ into two subsets  such that $i\in A$, $j\in B$, and $k$ separates $A$ and $B$. By (\ref{eq:rhoedge}), $\Sigma_{A,B}=\Sigma_{A,k}\Sigma_{k,B}$. We will show that $K_{A,B}=0$ and so in particular $K_{ij}=0$. Direct computations show that the matrix equation
\begin{equation}\label{eq:invsigtree}
		\begin{bmatrix}
		\Sigma_{A,A} & \Sigma_{A,k} & \Sigma_{A,k}\Sigma_{k,B}\\
		\Sigma_{k,A} & 1 & \Sigma_{k,B}\\
		\Sigma_{B,k}\Sigma_{k,A} & \Sigma_{B,k} & \Sigma_{B,B}
	\end{bmatrix}\cdot \begin{bmatrix}
		\mathbf A & \mathbf b  & \mathbf 0\\
		\mathbf b^T & c & \mathbf d^T \\
		\mathbf 0 & \mathbf d & \mathbf E
	\end{bmatrix}\;=\;I_n
\end{equation}
	has a solution with 
	$$
	\mathbf A\;=\;(\Sigma_{A,A}-\Sigma_{A,k}\Sigma_{k,A})^{-1},\qquad \mathbf b\;=\;-(\Sigma_{A,A}-\Sigma_{A,k}\Sigma_{k,A})^{-1}\Sigma_{A,k},
	$$
		$$
	\mathbf E\;=\;(\Sigma_{B,B}-\Sigma_{B,k}\Sigma_{k,B})^{-1},\qquad \mathbf d\;=\;-(\Sigma_{B,B}-\Sigma_{B,k}\Sigma_{k,B})^{-1}\Sigma_{B,k},
	$$
	$$
	c\;=\;1+\Sigma_{k,A}(\Sigma_{A,A}-\Sigma_{A,k}\Sigma_{k,A})^{-1}\Sigma_{A,k}+\Sigma_{k,B}(\Sigma_{B,B}-\Sigma_{B,k}\Sigma_{k,B})^{-1}\Sigma_{B,k}.
	$$
	Since the first term in the product in (\ref{eq:invsigtree}) is the matrix $\Sigma$, the second matrix must be $K$. We conclude that $K_{A,B}=0$.
\end{proof}

In this article we assume that the genericity condition of Theorem~\ref{th:STDrank} holds.

\begin{assumption}\label{mntool}
The matrix $\Sigma\in \mathcal M(G)$ is always to be generic, or equivalently, for every $A,B\subseteq V$, ${\rm rank}(\Sigma_{A,B})=\min\{|S|:\;S\mbox{ separates }A\mbox{ and }B\}$.
\end{assumption}

 This assumption gives us the following important result that translates small sets of covariance queries into information about the underlying concentration graph $\mathcal G(\Sigma)$.
\begin{lemma}\label{mntool3}
Under Assumption~\ref{mntool}, ${\rm rank}(\Sigma_{AC,BC})={\rm rank}(\Sigma_{A,B})$ if and only if $C$ is a subset of a minimal separator of $A$ and $B$ in $\mathcal G(\Sigma)$.
\end{lemma}

\noindent Here and throughout we use the convention of writing $A\cup B$ as $AB$ in subindices.

\medskip

\begin{proof}By Assumption~\ref{mntool}, ${\rm rank}(\Sigma_{AC,BC})$ is the size of a minimal separator of $A\cup C$ and $B\cup C$, and ${\rm rank}(\Sigma_{A,B})$ is the size of a minimal separator of $A$ and $B$. Since ${\rm rank}(\Sigma_{AC,BC})={\rm rank}(\Sigma_{A,B})$, there is a minimal separator of $A\cup C$ and $B\cup C$ that is also a minimal separator for $A$ and $B$. By construction, this separator contains $C$. \end{proof}
\begin{remark}
As mentioned in the introduction, vanishing partial correlations do not necessarily translate to conditional independence statements. In the tree case, conditional independence is implied not only for Gaussian and non-paranormal data but also for binary variables, or more generally, in  situations where the dependence of adjacent variables in the tree is linear; see \cite{ltm} for more details. 
\end{remark}

\subsection{Overview of the main results}\label{sec:overview}

Formulating simplified versions of our main results, we use the notation $\mathcal O_{\alpha}$ to denote that the complexity order contains a factor depending on parameters $\alpha$. Our first result studies computationally efficient ways to learn a tree.
\hypersetup{linkcolor=black}
\begin{customthm}{\ref{mmain}}[Simplified version]
	Suppose $\mathcal G(\Sigma)=T=([n],E)$ is a tree with $n$ vertices and maximum degree $\Delta(T)\leq d$.  
Then there is an algorithm that outputs the correct tree and, with probability at least $1-\epsilon$, works in time and query complexity $\mathcal O_{\epsilon ,d}(n\log ^2n)$.  
\end{customthm}
\hypersetup{linkcolor=red}
\noindent In Theorems~\ref{thm:lowerstar} and \ref{thm:lowerdary} we show that these bounds are essentially optimal and the dependence on the maximum degree is essential. 

Our second result is for graphs with small 2-connected components and small degree of the  block-cut tree; see Section~\ref{sec:treelike} for formal definitions.
\hypersetup{linkcolor=black}
\begin{customthm}{\ref{mmain2}}[Simplified version]
	Let $G=([n],E)$ be a graph whose largest 2-connected component has size at most $b$ and whose maximum degree of the block-cut tree is at most $d$. If $\Sigma\in \mathcal M(G)$ is generic, then there is an algorithm that outputs the correct graph and, with probability at least $1-\epsilon$, works in time and query complexity $\mathcal O_{\epsilon ,d,b}(n\log ^2n)$. \end{customthm}
	\hypersetup{linkcolor=red}

Our main result is presented in Sections~\ref{sec:boundedTWsep} and \ref{sec:boundedTW}. We propose a randomized algorithm that is able to recover efficiently the concentration graph $\mathcal G(\Sigma)$ 
as long as this graph has bounded  \emph{treewidth} and maximum degree. (In fact, the algorithm remains efficient when both parameters grow slowly with $n$.) Graphs with bounded treewidth form an important class of sparse graphs that have played a central
role in graph algorithms. The class of graphs with small treewidth includes series-parallel graphs, 
outerplanar graphs, Halin graphs, Apollonian networks, and many others, see \cite{bodlaender1998partial} for a general reference.
Treewidth has also been known to be an essential parameter in inference and structure recovery for graphical models \cite{chandrasekaran2012complexity,kwisthout2010necessity,wainwright2008graphical}. 
\hypersetup{linkcolor=black}
\begin{customthm}{\ref{th:mainmain}}[Simplified version]
	Let $G=([n],E)$ be a graph with treewidth at most $k$ and 
maximum degree at most $d$.
 If $\Sigma\in \mathcal M(G)$ is generic, then there is an algorithm that outputs the correct concentration graph and, with probability at least $1-\epsilon$, works in time and query complexity $\mathcal O_{\epsilon ,k,d}(n\log ^5 n)$. \end{customthm}
\hypersetup{linkcolor=red}

Actually, the algorithm we propose not only reconstructs the concentration graph $\mathcal G(\Sigma)$ but also it computes the precision matrix $K$. Since there are at most $kn$ edges in a graph with treewidth $k$, there is no contradiction with the stated computational 
complexity.

\section{Recovery of tree-like structures}\label{sec:treelike}

In this section we discuss in detail procedures for learning trees and graphs with small 2-connected components. A graph is \emph{2-connected} if for any vertex $v$, $G\setminus v$ is connected. 
If $V'\subseteq V$ is maximal, with respect to inclusion, such that $G[V']$ is 2-connected, then  $G[V']$ is a \emph{2-connected component} or \emph{block} of $G$.

{Theorem \ref{mmain2} below clarifies what we mean by ``small 2-connected components.''
  In particular, if the size $b$ of the largest 2-connected component is at most of the order $\log^{1/3}n$,
  then both the query and computational complexity of the proposed recovery algorithm are of the order $n\log^2 n$.
}  
 
 For a given graph $G$, let $\cB$ be the set of 2-connected components of $G$ and let $A$ be the set of {\it cut-vertices}, that is, vertices that belong to more than one 2-connected components. The \textit{block-cut tree} $\bc(G)$ of $G$ is a bipartite graph on $A\cup \cB$ where an edge between $a\in A$ and $B\in \cB$ exists if $a\in B$.  A block-cut tree is a tree by \cite[Theorem 4.4]{harary6graph}. {See Figure~\ref{fig:100} for an example of a graph and its block-cut tree.}
 \begin{figure}
 	\includegraphics[scale=.8]{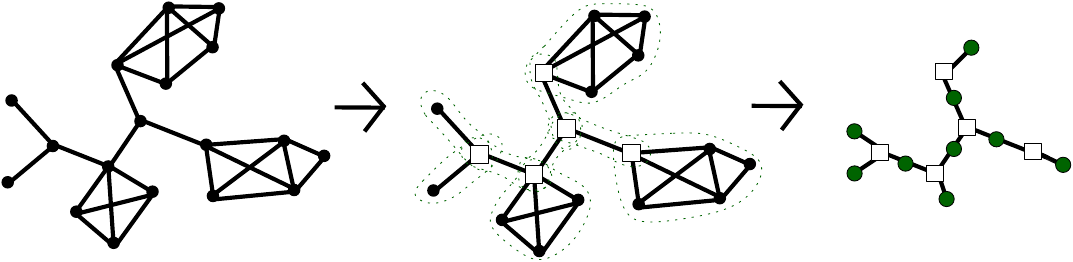}
 	\caption{{A graph, its cut-vertices (in squares) and its 2-connected components (circled by dotted curves), and the corresponding block-cut tree.}}\label{fig:100}
 \end{figure}

Both proposed learning procedures---one for recovering trees, another, more general one, for graphs with small 
$2$-connected components---are divide-and-conquer type algorithms. In both methods we first determine a cut
vertex that splits the graph into relatively small pieces, identify the pieces, and proceed recursively. 
Hence, the starting point of our analysis is to identify, at each step of the algorithm, a cut vertex (i.e., a separator 
of size one) that is balanced.

\subsection{Centrality and balanced separators}\label{sec:central}

Let $\cC^v$ be the set of connected components of $G\setminus v$ and define
\begin{equation}\label{eq:cv}
c(v)\;=\;\frac{1}{|V|-1}\max_{C\in \cC^v}|C|.	
\end{equation}
 Denote by $v^*$ a vertex that attains the minimum such value, that is, $v^*=\argmin_{v\in V} c(v)$. If $G$ is a tree,  $v^*$ is called a \emph{centroid}. It is a well-known fact (see \cite[Theorem 4.3]{harary6graph}, for instance) that a tree can have at most two centroids and $c(v^*)\leq \frac{1}{2}\frac{|V|}{|V|-1}$. 

In the first phase we efficiently find vertices with $c(v)\le \alpha$ for a fixed $\alpha<1$.  To that end, we introduce a measure of vertex centrality, called \emph{$s$-centrality} $s(v)$, that can be used as a surrogate for $c(v)$ and whose minimizer can be  approximated efficiently. For each $v\in V$, s-centrality is defined as
\begin{equation}s(v)\;=\;\frac{1}{(|V|-1)^2}\sum _{C\in \cC^v}|C|^2.\label{s-central}\end{equation}
We denote $v^\circ=\argmin_{v\in V} s(v)$.

\begin{lemma}\label{property}For every vertex $s(v) \leq c(v)\leq \sqrt{s(v)}$. Moreover, $s(v^\circ)\leq c(v^*)$. \end{lemma}
\begin{proof}
Let $v\in V$ and $\cC^v=\{C_1,\dots,C_m\}$. The first inequality follows from 
$$s(v)\;\leq\; \frac{1}{(|V|-1)^2}\sum _{j=1}^{m} |C_j| \max_{i} |C_i|\;=\;\frac{|V|-1}{(|V|-1)^2} \max_{i \in [m]} |C_i|\;=\;c(v).
$$
To show $c(v)\leq \sqrt{s(v)}$, consider the vector $p=(p_1,\ldots,p_m)$ with $p_i=\frac{|C_i|}{|V|-1}$. Then $c(v)=\|p\|_\infty$ and $s(v)=\|p\|_2^2$.
 The second inequality simply follows from the fact that $\|p\|_\infty\leq \|p\|_2$ for every $p\in \R^m$. To show the last inequality note that $s(v^\circ)\leq  s(v^*)$ by the optimality of $v^\circ$ and $s(v^*)\leq c(v^*)$ by the first proved inequality.
\end{proof}

\noindent The procedure ${\tt sCentral}$ outlined in Algorithm~\ref{sRecovery} finds, with high probability, a vertex $\hat{v}$ with $s(\hat{v})$ close to $s(v^\circ)$. For each vertex $v\in V$ the algorithm approximates $s(v)$ by randomly sampling a few pairs $u,w$ of vertices in $V\setminus \{v\}$ and checking if $v$ separates $u$ and $w$. By Lemma~\ref{lem:prodrho}, this can be accomplished by checking if $\Sigma_{uv}\Sigma_{vw}=\Sigma_{uw}\Sigma_{vv}$,
or equivalently, if $\det(\Sigma_{uv,vw})=0$. 

The algorithm outputs a vertex with smallest approximate value of $s(v)$.

\begin{algorithm}\label{sRecovery}
Parameter: $\kappa$\;  
 $\hat{s}(v):= 0$ for all $v\in V$\;
\For{all $v\in V$}{
\For{$i=1$  \KwTo $\kappa$}{ Pick $u, w$ uniformly at random in $V\setminus \{v\}$ \;  \If {$\det(\Sigma_{uv,vw})\neq 0$}{ $\hat{s}(v):=\hat{s}(v)+\frac{1}{\kappa}$\;}}}
Return $\arg\min _v \hat{s}(v) $\;
\label{s-algo}
\caption{${\tt sCentral}(V)$}
\end{algorithm}

\begin{proposition}\label{Recovery}Let $G=(V,E)$ be a graph. The time and query complexity of computing $\hat v={\tt sCentral}(V)$ are  $\mathcal{O}\left(|V|\kappa\right)$. Moreover, for any $\delta>0$
$$
\mathbb{P}\left( s(\hat{v} )\geq s(v^\circ)+2\delta \right) \;\;\leq\;\;  2|V|\exp \big( -2\delta^2\kappa\big).
$$
\end{proposition}

\begin{proof}
The time and query complexity are obtained in a straightforward way. For the second statement note that, for every $v\in V$, $\kappa\hat{s}(v)$ is a binomial random variable with mean $\kappa s(v)$. 
Hence, by Hoeffding's inequality and the union bound, we obtain
\begin{equation*}
\mathbb{P}[\max _v |\hat{s}(v)-s(v)|\geq \delta ]\;\leq\; 2|V|\exp \big( -2\delta^2\kappa\big) .
\end{equation*}
For $\hat{v}=\arg \min_{v} \hat{s}(v)$ after running Algorithm~\ref{s-algo} \begin{eqnarray*}
\mathbb{P}[s(\hat{v} )\geq s(v^\circ)+2\delta ] &\leq & \mathbb{P}[s(\hat{v}) -\hat{s}(\hat{v})+\hat{s}(v^\circ)-s(v^\circ)\geq 2\delta ] \\  &\leq & \mathbb{P}[\max _v\{ |\hat{s}(v)-s(v)|\}\geq \delta ]\leq 2|V|\exp \big( -2\delta^2\kappa\big),
\end{eqnarray*} 
where the second inequality follows from the fact that $(\hat s(v^\circ)-s(v^\circ))-(\hat s(\hat v)-s(\hat v))\leq 2\max _v|\hat s(v)-s(v)|$. 
\end{proof}
 
We now show that the above procedure finds a good splitting vertex with high probability. 
\begin{proposition}\label{prop:big}
	For any graph $G=(V,E)$, if  $s(v)< s(v^\circ)+2\delta$ then $$c(v)\;<\;\sqrt{s(v^\circ)+2\delta}\;\leq\;\sqrt{c(v^*)+2\delta}.$$ 
	In particular, if $G=(V,E)$ is a tree with $|V|\geq 4$, and $\delta< \frac{1}{6}$, then $c(v)<1$.\end{proposition}
\begin{proof}
Follows from Lemma~\ref{property} and the fact that for trees $c(v^*)\leq \frac{1}{2}\frac{|V|}{|V|-1}\leq \frac{2}{3}$. 
\end{proof}

Proposition~\ref{Recovery} and Proposition~\ref{prop:big} imply the following result, by fixing $\delta=1/8$. 
\begin{corollary}\label{cor:big}
	If $G=(V,E)$ is a tree, $|V|\geq 4$, and $\hat{v}={\tt sCentral}(V)$ then $$\mathbb{P}\left(c(\hat v)>\sqrt\frac{{11}}{12}\right)\;\;\leq\;\;  2|V|\exp(-\kappa/32).$$
\end{corollary}

In Proposition~\ref{prop:big} we used the fact that for trees $c(v^*)\leq \frac{2}{3}$ if $|V|\geq 4$. In  general, for graphs with small 2-connected components, we rely on the following result.
\begin{lemma}\label{newc}
Suppose $G=(V,E)$ is a connected graph. Let $d$ be the maximum degree of the block-cut tree, and let $b$ be the size of the largest 2-connected component of $G$. If $|V|> db$ then 
$$c(v^*)\;\;\leq\;\; 1-\frac{1}{2d}.$$
\end{lemma}
 \begin{figure}
 	\centering\includegraphics{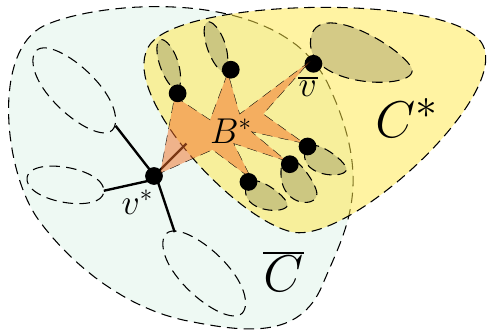}
 	\caption{Illustration of the proof of Lemma~\ref{newc}}\label{fig:lemma4}
 \end{figure}
\begin{proof}
If $|V|>db$ then $b<|V|$ and therefore $G$ has a cut vertex. Let $C^*$ be the largest component in $\cC^{v^*}$ and let $B^*$ be the 2-connected component of $G$ such that $v^*\in B^*$ and $B^*\setminus \{v^*\}\subseteq C^*$; c.f. Figure~\ref{fig:lemma4}. We can assume that there exists a cut vertex in $B^*\setminus \{v^*\}$, because otherwise $B^*=C^*$ and clearly $|V|\leq db$. For each cut vertex $v\neq v^*$ in $B^*$, let $k_v$ be the number of all vertices in the union of all the connected components of $\cC^v$ excluding the component containing $B^*$ (grey blobs in Figure~\ref{fig:lemma4}). By construction, $\sum_{v\neq v^*} k_v=|C^*|-|B^*|$ and there are at most $d-1$ such vertices. If $\overline v=\arg\max_{v\neq v^*} k_v$, then $k_{\overline v}\geq \frac{1}{d-1}(|C^*|-|B^*|)$. Denote by $\overline C$ the largest connected component of $G\setminus\{\overline v\}$. By optimality of $v^*$, $\overline C$ must be equal to the connected component  containing the complement of $C^*$. Thus we have
\begin{eqnarray*}
|C^*|\;\leq\; |\overline C|&\leq & (|V|-1-|C^*|)+|B^*|+(|C^*|-|B^*|-k_{\overline v})\leq\\
&\leq &	(|V|-1-|C^*|)+|B^*|+\frac{d-2}{d-1}(|C^*|-|B^*|)
\end{eqnarray*}
Now simple algebra and the fact that $|B^*|\leq b$ allow us to rewrite this inequality as 
$$c(v^*)\;\leq\; 1-\frac{1}{d}\left(1-\frac{b}{|V|-1}\right).$$ To prove our inequality, it is enough to show that $\frac{b}{|V|-1}\leq\frac{1}{2}$, which is obvious if $|V|\geq db+1$.\end{proof}

Proposition~\ref{prop:big} and Lemma~\ref{newc}  imply that whenever $s(v)<s(v^\circ)+2\delta$ and $\delta<\frac{1}{4d}$ then $c(v)<1$. 
By choosing $\delta=\frac{1}{8d}$, the following corollary follows from Proposition~\ref{Recovery}. 
\begin{corollary}\label{cor:big2}
		Let $G=(V,E)$ be a graph with $|V|> db$ and let $\hat{v}={\tt sCentral}(V)$. Then
		$$\mathbb P\left(c( \hat{v})> \frac{\sqrt{4d-1}}{\sqrt{4d}}\right)\;\;\leq\;\;  2|V|\exp\left(-\frac{\kappa}{32d^2}\right).$$
\end{corollary}

\subsection{Recovering a tree}\label{sec:tree}

In this section we present procedure {\tt ReconstructTree} (Algorithm~\ref{tmain}), which efficiently recovers the structure of the tree $T=\mathcal G(\Sigma)$. We start the procedure by running ${\tt ReconstructTree}([n])$. The algorithm updates an edge set $\hat E$ that is initiated as $\hat E=\emptyset$. At each call, if $|V|>1$, $V$ gets partitioned into sets $V_1,\dots ,V_m$ by procedure {\tt ComponentsTree} (Algorithm~\ref{split}) and the edge set $\hat E$ gets updated. Then, {\tt ReconstructTree} recurses into all the generated sets.

${\tt ComponentsTree}(V)$ first picks a central vertex $w={\tt sCentral}(V)$. It then sorts, in descending order, the absolute values of the pairwise correlations $\rho _{uw}=\frac{\Sigma_{uw}}{\sqrt{\Sigma_{uu}\Sigma_{ww}}}$, where $u\in V\setminus \{w\}$, and places them in an ordered list $B$. For every vertex $u$ in the sorted list, the algorithm checks whether $u$ is separated by $w$ from any of the known neighbours $v$ of $w$, or equivalently, if $\det(\Sigma_{uw,vw})=0$; c.f. Section~\ref{sec:central}. If this happens, then it adds $u$ to the connected component where $v$ belongs after removing $w$. If this does not happen, then a new connected component is registered corresponding to the vertex $u$ and the edge $uv$ is added in $\hat{E}$. In the end, ${\tt ComponentsTree}$ returns the vertex sets of all such connected components $V_1,\ldots,V_m$. The edges between $w$ and each of the $m$ neighbours in the $m$ connected components are added to $\hat E$.

\begin{algorithm}
 \If{$|V|>1$}{$V_{1},\dots,V_{m}\leftarrow {\tt ComponentsTree}(V)$\; \For{\upshape $i$ from $1$ to $m$}{ {\tt ReconstructTree}$(V_{i})$; } }
\vspace*{.1cm}
\label{tmain}
\caption{${\tt ReconstructTree}(V)$}
\end{algorithm}

\begin{algorithm}

  $w\gets {\tt sCentral}(V )$\; $N\gets \emptyset$\; 
Sort $|\rho_{uw}|$ for $u\in V\setminus \{w\}$ in decreasing order and put them in list $B$\;
  \For {\upshape every $u$ \upshape  in the order of $B$}{
$t:= true$\;
  \For{all $v\in N$}{\If{ $\det(\Sigma_{uw,vw})\neq 0$ }{$V_{v}\leftarrow V_{v}\cup\{u\}$;\,\, $t:= false$\;}}
\If{t=true}{$\hat E\leftarrow \hat E\cup \{uw\}$\; 
$N\gets N\cup\{u\}$, $V_{u}\gets \{u\}$\;
}
}
Return all  $V_{u}$ for $u\in N$\;

\label{split}
\caption{${\tt ComponentsTree}(V)$}
\end{algorithm}

\begin{proposition}
Algorithm~\ref{tmain} is correct, that is, if $\mathcal G(\Sigma)$ is a tree $T$
then ${\tt ReconstructTree}([n])$ gives $\hat E=E(T)$.\end{proposition}

\begin{proof}
After every call of {\tt ComponentsTree}$(V)$ it holds that $\bigcup_{i=1}^m V_{i}=V\setminus \{w\}$, hence ${\tt ReconstructTree}([n])$ always terminates.

For every $w\in [n]$, either  $w={\tt sCentral}(V)$ in one of the calls of {\tt ComponentsTree} (call such a vertex central), or $\{w\}$ is one of the components $V_1,\ldots, V_m$ that are returned by ${\tt ComponentsTree}(V)$ (call such a vertex terminal). 

Initially the algorithm picks a vertex $w={\tt sCentral}([n])$, which induces the partition of $[n]\setminus \{w\}$, $\cC^{w}=\{V_1,\ldots,V_m\}$. The vertices $u\in [n]\setminus \{w\}$ are examined in descending order with respect to $|\rho_{uw}|$.  Let $v\in N(w)$ be adjacent to $w$ and let $u$ be any other other vertex in the same connected component $C\in \cC^w$ as $v$. Then $v$ separates $u$ and $w$ and, in particular, by Lemma~\ref{lem:prodrho}, $|\rho_{vw}|>|\rho_{uw}|$. This shows that, for any $C\in \cC^w$, the vertex $v$ in $C$ which is a neighbour of $w$ comes earlier in the order specified in the algorithm than any other vertex in $C$. Hence, $v\in N$ (c.f. Algorithm~\ref{split}) and for all other $u\in C$ it holds that $\det (\Sigma_{uw,vw})\neq 0$. This shows that in the first call of {\tt ComponentsTree} the algorithm:
\begin{enumerate}
	\item[(i)] adds to $\hat E$ the edges between the central vertex $w$ and its neighbours in $T$, 
	\item[(ii)] assigns all vertices to their connected components in $\cC^w$.
\end{enumerate}

Since each $G[V_i]$ is a tree, the same argument can be applied to subsequent calls of {\tt ReconstructTree}. Hence, by induction, these two properties hold at any call of the algorithm {\tt ComponentsTree}. In particular, $\hat E\subseteq E(T)$. To show the opposite inclusion first note that if $uv\in E(T)$ and $u$ or $v$ is central, then $uv\in \hat{E}$ by (i). Moreover, if $u,v$ are both terminal, then there is some call of {\tt ComponentsTree} that places them in different sets $V_i$. Then by (ii), there is no edge $uv$ in $E$.
\end{proof}

Subroutine {\tt sCentral} is a probabilistic component of the algorithm that is essential to obtain good complexity bounds. 

\begin{theorem}\label{mmain}
Suppose $\mathcal G(\Sigma)$ is a tree $T=([n],E)$ with maximum degree $\Delta(T)\leq d$. 
Fix $\epsilon <1$ and define $\kappa=\lceil 32\log\!\left( \frac{2n^2}{\epsilon } \right)\rceil$ to be the parameter of Algorithm~\ref{s-algo}. 
Then, with probability at least $1-\epsilon $, Algorithm~\ref{tmain} requires time and queries of order
$$\mathcal{O} \left( n\log(n) \max\left\{\log\left(\frac{n}{\epsilon}\right),d\right\} \right).$$
\end{theorem}
\begin{proof}
First we analyze the complexity of one call of ${\tt ComponentsTree}(V)$. By Proposition~\ref{Recovery}, the call of ${\tt sCentral}(V)$ takes time and queries of the order $\mathcal{O}\big(|V|\kappa\big)$. We then query $|V|$ pairwise correlations and sort them, which takes time  $\mathcal{O}(|V|\log |V|)$. Partitioning $V$ into sets $V_i$ takes time and queries both of order $\mathcal{O}(d|V|)$ since $|N|\leq d$. For all calls of {\tt ReconstructTree}$(V_i)$ in a recursion level (i.e., distance from the first {\tt ReconstructTree} call in the recursion tree), it holds that $V_i\cap V_j=\emptyset $. Hence, in each recursion level the time complexity is of order $\mathcal{O}(n\log n+ n\kappa+nd )$ and the query complexity is of order $\mathcal{O}(n\kappa+nd )$.

Assume first that $\hat v={\tt sCentral}(V)$ satisfies $c(\hat v)\leq \alpha:=\sqrt{11/12}$ in each call with $|V|\geq 4$. In this case the recursion depth is at most $\log_{1/\alpha}(n)+4$ and, overall, the algorithm has time  complexity $\mathcal{O} \big( n\log(n) (\log(n) +\kappa+d) \big) $ and query complexity $\mathcal{O} \big( n\log(n)  (\kappa+d) \big) $. Since $\kappa=\mathcal O(\log(n/\epsilon)+\log d)$, the announced bounds follow.

It remains to show that, with the given choice of $\kappa$, with probability at least $1-\epsilon$ we get that $c(\hat v)\leq \alpha$ in each call with $|V|\geq 4$. By Corollary~\ref{cor:big}, in a single call the probability that $c(\hat v)> \alpha$ is at most $2|V|\exp (- \kappa/32)$, which is further bounded by $2n\exp (- \kappa/32)$. As {\tt ReconstructTree}$([n])$ runs, the procedure {\tt sCentral} is called at most $n$ times, which is the total number of available vertices. From the union bound, the probability that in at least one call $\hat v={\tt sCentral}(V)$ satisfies $c(\hat v)> \alpha$ is at most $2n^2 \exp (- \kappa/32)$. Demanding the latter to be at most $\epsilon $, we obtain the indicated value for $\kappa$ and the desired result.\end{proof}

\subsection{Graphs with small blocks}\label{sec:smallblocks}

We now present an algorithm that recovers concentration graphs with small 2-connected components and small maximum degree of the block-cut tree. The procedure ${\tt ReconstructSB}$ (Algorithm~\ref{tlmain}) takes as input a vertex set $V$ and, like in the tree case, updates the global variable $\hat E$, which is initially set as $\hat E:=\emptyset$. If $V$ is small enough, that is, $|V|\leq db$, then the algorithm reconstructs the induced graph over $V$ by directly inverting the matrix $\Sigma_{V}$. Otherwise it calls ${\tt ComponentsSB}$, which first finds a vertex $w={\tt sCentral}(V)$ and  returns sets $C\cup \{w\}$ for all $C\in \cC^w$. This part of the algorithm is similar to ${\tt ComponentsTree}(V)$, but the edges incident with $w$ are not recovered at this stage. %
\begin{algorithm}
 \uIf{$|V|> db$}{$V_{1},\dots,V_{m}\leftarrow {\tt ComponentsSB}(V)$\;
\For{$i$ \upshape from $1$ to $m$}{ 
		${\tt ReconstructSB}(V_{i})$\;	
		} }
\Else{$\hat E \gets  \hat{E} \cup E(\mathcal G(\Sigma_{V,V}))$;}
\label{tlmain}
\caption{${\tt ReconstructSB}(V )$}
\end{algorithm}
\begin{algorithm}
  $w\leftarrow {\tt sCentral}(V)$\;  
  $N:= \emptyset $\tcp*{contains one vertex from each $C\in \cC^w$}
  \For{\upshape all $u\in V\setminus \{w\}$}{
  	\If{\upshape there exists  $v\in N$, $\det(\Sigma_{uw,vw})\neq 0$}{$V_{v}\gets V_{v}\cup \{u\}$\; }\Else{$V_{u}\leftarrow \{u\}$\; $N\leftarrow N\cup \{u\}$;} }
	 \Return{$V_u\cup \{w\}$ for all $u\in N$\;}
\label{tlsplit}
\caption{${\tt ComponentsSB}(V)$}
\end{algorithm}

\begin{proposition}
Algorithm~\ref{tlmain} is correct, that is, if $\Sigma\in \mathcal M(G)$ and $\Sigma$ is generic then ${\tt ReconstructSB}([n])$ gives $\hat{E}=E(G)$.
\end{proposition}
\begin{proof} Assume we are on the first call of {\tt ReconstructSB}. If $n\leq db$ then the algorithm outputs $\mathcal G(\Sigma)$, which, by definition, is the correct graph. If $n>db$, then $G$ contains a cut vertex, so with probability 1 a cut vertex $w$ will be found by ${\tt sCentral}([n])$. Let $C_1,\ldots, C_m$ be the connected components of $G\setminus \{w\}$. The sets $V_i$ produced by this call of ${\tt sCentral}([n])$ correspond to the sets $C_1\cup \{w\},\ldots, C_m\cup \{w\}$. This is clear by Lemma~\ref{mntool3}: a vertex $u$ belongs to the same connected component as $v$ in $G\setminus \{w\}$ if and only if $w$ does not separate $u$ and $v$ in $G$. Note also that for any $A,B\subset V_i$ any minimal separator of $A$ and $B$ is contained in $V_i$. In particular, by Theorem~\ref{th:STDrank},  
\begin{enumerate}
	\item[(i)] For every $V_i$, the edge-set of $G[V_i]$ is the same as the edge-set of the graph of the marginal distribution, $\mathcal G(\Sigma _{V_i})$.
\end{enumerate}
By induction, statement (i) holds for every call of {\tt ReconstructSB}.
\end{proof}

\begin{theorem}
\label{mmain2}
Let $G=([n],E)$ be a graph with maximum degree of the block-cut tree bounded by $d$, and let $b$ be the size of the largest 2-connected component. Fix $\epsilon <1$ and define $\kappa=\lceil32 d^2\log\!\left( \frac{2dn}{\epsilon } \right)\rceil$ to be the parameter of Algorithm~\ref{s-algo}. If $\Sigma\in \mathcal M(G)$ and $\Sigma$ is generic, then, with probability at least $1-\epsilon $, Algorithm~\ref{tlmain} runs with query and time complexity of order
\[\mathcal{O}\bigg(  d^3n\log(n) \bigg(\log\left(\frac{n}{\epsilon}\right)+b^2\bigg)\bigg) \quad \text{and} \quad \mathcal{O}\bigg(  d^3n\log(n)\bigg(\log\left(\frac{n}{\epsilon}\right)+db^3\bigg)\bigg)\]
\end{theorem}
\begin{proof}
First we analyze the complexity of one call of ${\tt ComponentsSB}(V)$. If $|V|\leq  db$ then $\mathcal G(\Sigma_{V,V})$ is obtained. Since  matrix inversion takes at most cubic time, the time and queries required are $\mathcal{O}(d^3 b^3)$ and $\mathcal{O}(d^2b^2)$ respectively. 
If $|V|> db$ in ${\tt ReconstructSB}(V)$, then ${\tt ComponentsSB}(V)$ is called, which calls ${\tt sCentral}(V )$. By Proposition~\ref{Recovery}, the call of ${\tt sCentral}(V)$ takes time and queries of the order $\mathcal{O}\big(|V|\kappa\big)$. The latter provides the splitting vertex $w$  and then $V$ is split into at most $d$ sets $U_i=V_i\cup \{w\}$. This last step takes $\mathcal O(|V|d)$ queries and time. Hence, each call ${\tt ReconstructSB}(V)$ requires 
$$\mathcal{O}\bigg( |V|\kappa+|V|d+d^2b^2\bigg) \quad \text{ and}\quad  \mathcal{O}\bigg( |V|\kappa+|V|d+d^3b^3\bigg) $$
queries and time, respectively.

Let $U_{1},\dots,U_{r}$ be the sets on which the algorithm recurses on the $i$-th level of the recursion tree. By construction, these sets are not disjoint and, for each $v\in V$, at most $d$ copies of it are created during the algorithm. Hence, in each recursion level there are at most $nd$ vertices, implying that $\sum_{i=1}^r |U_{i}|\leq nd$ and that $r\leq nd$. 
Using the complexity bounds for a single call of ${\tt ReconstructSB}$, we get that any recursion level in the recursion tree requires
$$\mathcal{O}\bigg( nd\kappa+nd^3b^2\bigg) \quad \text{ and}\quad  \mathcal{O}\bigg( nd\kappa+nd^4b^3\bigg) $$
queries and time, respectively. 

Assume first that $\hat v={\tt sCentral}(V)$ satisfies $c(\hat v)\leq \alpha:=\sqrt{4d-1}/\sqrt{4d}$ in each call. In this case the recursion depth is at most $\log _{1/\alpha}(n)$ and  overall the algorithm requires 
$$\mathcal{O}\bigg(\frac{dn\log (n)}{\log\left(1/\alpha\right)} \left(\kappa+d^2b^2\right)\bigg) \quad \text{ and}\quad  \mathcal{O}\bigg(\frac{dn\log (n)}{\log\left(1/\alpha \right)}\left( \kappa+d^3b^3\right)\bigg) $$
queries and time, respectively. Since $\kappa=\mathcal O(d^2 \log(n/\epsilon))$ and $1/\log(1/\alpha)\leq 20d$, we obtain the expressions in the statement of the theorem. 

It remains to show that with the given choice of $\kappa$, with probability at least $1-\epsilon$, we get that $c(\hat v)\leq \alpha$ in each call. By Corollary~\ref{cor:big2}, in a single call the probability that $c(\hat v)> \alpha$ is at most $2|V|\exp \left(- \frac{\kappa}{32d^2}\right)$. As {\tt ReconstructSB}$([n])$ runs, the procedure {\tt sCentral} is called at most $dn$ times. From the union bound, the probability that in at least one call $\hat v={\tt sCentral}(V)$ satisfies $c(\hat v)> \alpha$ is at most $2dn \exp (-\kappa/(32d^2))$. This is at most $\epsilon $ for the indicated value of $\kappa$.
\end{proof}

\subsection{Lower bounds}\label{sec:101}

In this section we show that the result of Theorem~\ref{mmain} is optimal up to logarithmic factors, 
in the sense that one cannot reconstruct trees with maximum degree $d$ with less than $\Omega(dn)$ 
covariance queries.

Let $\X$ be the class of $n\times n$ covariance matrices whose
concentration graph is a tree. We write $T(\Sigma)$ for the tree induced by 
$\Sigma \in \X$.
We also denote by $\X_d$ the class of covariance matrices whose concentration graph is a tree
of maximum degree bounded by $d$. In our construction we use the characterization of the class $\X$ given in Lemma~\ref{lem:prodrho}.

We first prove that any algorithm 
that recovers the correct tree (without any restriction on the maximum degree)
 needs to access the covariance oracle $\Omega(n^2)$ times. 

In order to 
formalize such a statement,  let $\cA_k$ be the class of all randomized adaptive algorithms that 
query the covariance oracle at most $k$ times (\textit{adaptive} means that the algorithms are allowed to dynamically choose their queries). 
An algorithm $A\in \cA_k$
outputs the tree $\cT(A)$. The probability of error of algorithm $A$ for $\Sigma\in \X$
is denoted by
\[
    P(A,\Sigma) = \PROB\left\{ \cT(A) \neq T(\Sigma) \right\}~,
\]
where the probability is with respect to the randomization of the algorithm $A$.
The quantity of interest is the \emph{minimax risk}
\[
   R(\cA_k,\X)= \inf_{A \in \cA_k}\sup_{\Sigma \in \X} P(A,\Sigma)~.
\]
$R(\cA_k,\X)$ expresses the worst-case probability of error of the best algorithm 
that takes at most $k$ covariance queries. 

Theorem \ref{mmain} implies that there exists a constant $c>0$ such that, for every $\epsilon >0$, we have $R(\cA_k,\X_d)\le \epsilon$
whenever $k> cn \log(n)\left(d+\log(n/\epsilon)\right)$. In this section we prove that this upper bound is tight
up to logarithmic factors. 

We start with the case $d=n-1$ (i.e., no restriction on the maximum degree) since this simpler case
already contains the  main ideas. The lower bound for $R(\cA_k,\X_d)$ follows by a simple adjustment.

\begin{theorem}
\label{thm:lowerstar}
For all $k\le \binom{n}{2}$,
\[
    R(\cA_k,\X) \ge \frac{1}{2} - \frac{k}{(n-1)^2}~.
\]
In particular,     $R(\cA_k,\X) \ge 1/2-o(1)$ whenever $k=o(n^2)$.
\end{theorem}

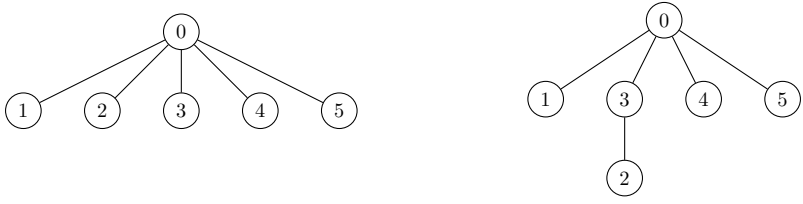
\begin{figure}\centering
\begin{tikzpicture}[scale=0.7, every node/.style={scale=0.7},level/.style={sibling distance=15mm/#1}]
  \node (a) at (-2,-3) { };
\node [circle,draw] (z){$0$}
    child {node [circle,draw] (b) {$1$}
    }
    child {node [circle,draw] (g) {$2$}
    }
    child {node [circle,draw] (g) {$3$}
    }
        child {node [circle,draw] (g) {$4$}
    }
            child {node [circle,draw] (g) {$5$}
    }; 
\end{tikzpicture}
\hspace{2cm}
\begin{tikzpicture}[scale=0.7, every node/.style={scale=0.7},level/.style={sibling distance=15mm/#1}]
\node [circle,draw] (z){$0$}
    child {node [circle,draw] (b) {$1$}
    }
    child {node [circle,draw] (g) {$3$}
    child {node [circle,draw] (f) {$2$}
    }
    }
        child {node [circle,draw] (g) {$4$}
    }
                child {node [circle,draw] (g) {$5$}
    }
;
\end{tikzpicture}
	\caption{Illustration of the construction in the proof of Theorem~\ref{thm:lowerstar} with $n=6$, $\{I,J\}=\{2,3\}$, and $B=0$ (left), $B=1$ (right).}\label{fig:lowerstar}
\end{figure}
\begin{proof}
In order to prove the lower bound, we define a probability distribution $\cD$ on the set $\X$
and write
\[
   R(\cA_k,\X) \ge \inf_{A \in \cA_k} \EXP_{\Sigma \sim \cD} P(A,\Sigma)~.
\]
Next we specify how a random symmetric matrix $\Sigma$, distributed according to $\cD$, is generated.
$\Sigma$ is defined by a collection of independent random variables: let $B$ be a Bernoulli random variable
with parameter $1/2$, let  $U_1,\ldots,U_{n-1}$
be independent random variables, uniformly distributed on $[0,1]$, and let $I,J$ be different indices
in $[n-1]$, uniformly distributed over all $(n-1)(n-2)$ such pairs.
Then
 $\Sigma=\Sigma(B,U_1,\ldots,U_{n-1},I,J)$ is defined as follows; c.f. Figure~\ref{fig:lowerstar}. (We index the $n$ columns and rows from $0$ to $n-1$.)

\noindent $\bullet$
$\Sigma_{i,i}=1$ for all $i=0,\ldots,n-1$.

\noindent $\bullet$
Regardless of $B,I,J$, we have $\Sigma_{0,i}=U_i$ for all $i=1,\ldots,n-1$.

\noindent $\bullet$
If $B=0$, then $\Sigma_{i,j}= U_iU_j$ for all $i,j\in \{1,\ldots,n-1\}$, $i\neq j$.
Note that in this case, by Lemma~\ref{lem:prodrho}, the concentration graph is a star with vertex $0$ as a center
(and therefore indeed $\Sigma \in \X$).

\noindent $\bullet$
If $B=1$, then $\Sigma_{i,j}= U_iU_j$ for all $i,j\in \{1,\ldots,n-1\}$ such that $i\neq j$
and $\{i,j\}\neq \{I,J\}$. Moreover, $\Sigma_{I,J}=\min(U_I,U_J)/\max(U_I,U_J)$.
In this case, again by Lemma~\ref{lem:prodrho}, the concentration graph is a tree in which vertex $0$ has degree $n-2$,
every vertex $i\notin \{I,J\}$ has degree $1$ and is attached to vertex $0$,
and vertices $0,I,J$ form a path such that, if $U_I<U_J$, then $J$ is the middle vertex
and if $U_I>U_J$, then $I$ is the middle vertex.

Clearly, regardless of what the algorithm $A$ is, it is unable to distinguish
between $\Sigma(0,U_1,\ldots,U_{n-1},I,J)$ and $\Sigma(1,U_1,\ldots,U_{n-1},I,J)$
before the entry $\Sigma_{I,J}$ is queried. (No other entry of $\Sigma$ provides
any information about $\Sigma_{I,J}$.) In other words, if $B$, $U_1,\ldots,U_{n-1}$ and $I,J$ are fixed and $\Sigma=\Sigma(B,U_1,\ldots,U_{n-1},I,J)$ then 
$$
P(A,\Sigma)\;\geq \; \frac{1}{2}\,\EXP\left(\IND_{ (I,J) \ \text{is not queried}}|(I,J)\right).
$$
 Thus, for any algorithm $A$,
\[
\EXP_{\Sigma \sim \cD} P(A,\Sigma) \;\ge\; \frac{1}{2\binom{n-1}{2}} \sum_{\{i,j\} \subset [n-1]: i\neq j} \EXP\left(\IND_{ (i,j) \ \text{is not queried}}|(I,J)=(i,j)\right)
\;\ge\; \frac{\binom{n-1}{2}-k}{2\binom{n-1}{2}} 
\]
by symmetry, proving the theorem.
\end{proof}

\begin{figure}\centering
\begin{tikzpicture}[scale=0.5, every node/.style={scale=0.5},level distance=1.5cm,
  level 1/.style={sibling distance=3cm},
  level 2/.style={sibling distance=1cm, draw=black!60!green}]
  \node (a) at (-2,-4.5) { };
\node [circle,draw] (z){$1$}
[sibling distance=9cm]
    child {node [circle,draw] (b) {$12$}
    child {node [circle,draw] (t) {$3$}
    }
    child {node [circle,draw] (w) {$2$}
    }
    child {node [circle,draw] (q) {$4$}
    }
    }
    child {node [circle,draw] (g) {$5$}
    child {node [circle,draw] (t) {$6$}
    }
    child {node [circle,draw] (w) {$7$}
    }
    child {node [circle,draw] (q) {$8$}
    }
    }
    child {node [circle,draw] (g) {$9$}
       child {node [circle,draw] (t) {$10$}
    }
    child {node [circle,draw] (w) {$11$}
    }
    child {node [circle,draw] (q) {$13$}
    }
    }
; 
\end{tikzpicture}
\hspace{1cm}
\begin{tikzpicture}[scale=0.5, every node/.style={scale=0.5},level distance=1.5cm,
  level 1/.style={sibling distance=3cm},
  level 2/.style={sibling distance=1cm, draw=black!60!green}]
  \node (a) at (-2,-3) { };
\node [circle,draw] (z){$2$}
[sibling distance=9cm]
    child {node [circle,draw] (b) {$1$}
    child {node [circle,draw] (t) {$3$}
    }
    child {node [circle,draw] (w) {$13$}
    }
    child {node [circle,draw] (q) {$4$}
    }
    }
    child {node [circle,draw] (g) {$5$}
    child {node [circle,draw] (w) {$7$}
        child {node [circle,draw] (t) {$6$}
    }
    }
    child {node [circle,draw] (q) {$8$}
    }
    }
    child {node [circle,draw] (g) {$9$}
       child {node [circle,draw] (t) {$10$}
    }
    child {node [circle,draw] (w) {$11$}
    }
    child {node [circle,draw] (q) {$12$}
    }
    }
; 
\end{tikzpicture}
\hspace{1cm}
\begin{tikzpicture}[scale=0.5, every node/.style={scale=0.5},level distance=1.5cm,
  level 1/.style={sibling distance=3cm},
  level 2/.style={sibling distance=1cm, draw=black!60!green}]
  \node (a) at (-2,-3) { };
\node [circle,draw] (z){$5$}
[sibling distance=9cm]
    child {node [circle,draw] (b) {$1$}
    child {node [circle,draw] (t) {$3$}
    }
    child {node [circle,draw] (w) {$2$}
    }
    child {node [circle,draw] (q) {$4$}
    }
    }
    child {node [circle,draw] (g) {$7$}
     child {node [circle,draw] (t) {$6$}
     child {node [circle,draw] (w) {$11$}}}
     child {node [circle,draw] (q) {$8$}}
    }
    child {node [circle,draw] (g) {$9$}
       child {node [circle,draw] (t) {$10$}
           child {node [circle,draw] (q) {$13$}
    }
    }
    child {node [circle,draw] (w) {$12$}
    }
    }
; 
\end{tikzpicture}
	\caption{Three trees made from a ternary tree, by moving at most one lower level leaf from its parent to one of its siblings.}\label{fig:ternary}
\end{figure}
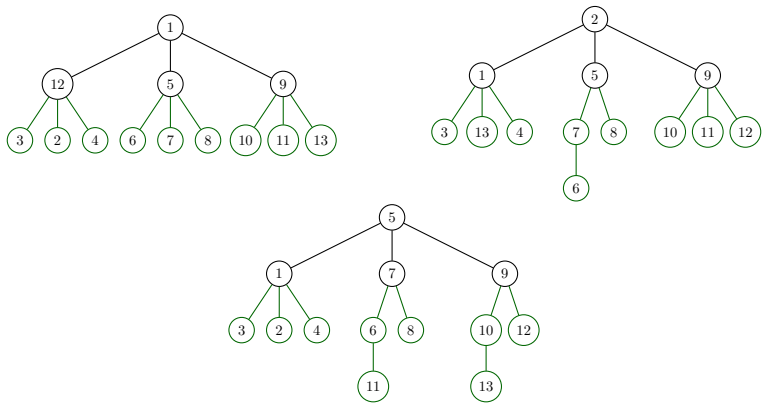

For the class of covariance matrices $\X_d$ whose concentration graph is a tree with maximum degree
bounded by $d$, we have the following bound. Its proof is similar to that of Theorem \ref{thm:lowerstar}.
To avoid repetitions, we do not detail the proof. The only difference is that the class of trees that support
the distribution $\cD$ now includes the complete $d$-ary tree of height $h$ of $n=(d^{h+1}-1)/(d-1)$ vertices
and its modifications such that, in each $d$-ary branch at the leaf level, one can remove a leaf and
attach it to another one of the same branch (see Figure~\ref{fig:ternary} for such instances, made from a ternary tree).

\begin{theorem}\label{thm:lowerdary}
\[
    R(\cA_k,\X_d) \ge \frac{1}{2}\left(1-o(1)\right)
\]
whenever $k=o(nd)$.
\end{theorem}

\section{Separators in bounded treewidth graphs}\label{sec:boundedTWsep}

In the next two sections we deal with the main results of the paper. We show that a large and important 
class of sparse concentration graphs can be reconstructed efficiently with $O(n\polylog n)$ correlation queries. The class
includes all graphs with bounded treewidth and bounded maximum degree.

The algorithm we propose is a divide-and-conquer
algorithm. The main idea is that, once one finds a small set of vertices (a separator) whose removal decomposes
the graph into small connected components, and these components are identified, one may recurse in these components.
The nontrivial task is to find such separators efficiently. 

Our algorithm starts by taking a random sample $W$ of the vertices, of size proportional to the treewidth of $G$. Then we find a separator of $W$ of size at most $k+1$ that splits the vertices of $W$ into two sets of comparable size. We argue that, with high probability, such a separator exists. We also prove that this separator is a balanced separator of the entire vertex set. Removal of this separator decomposes the graph into connected components of significantly reduced size. We identify these components using a  linear number of queries. Then the algorithm recurses into each of the components. In this section we discuss the first splitting of the graph. How to subsequently recurse into the smaller subsets is described in detail in Section~\ref{sec:boundedTW}.

A \emph{tree decomposition} of a graph $G=(V, E)$ is a tree $T$ with vertices $B_1,\ldots,B_m$ where $B_i\subseteq V$ satisfy
\begin{enumerate}
\item The union of all sets $B_i$ equals V. 
\item If $B_i$ and $B_j$ both contain $v$, then all vertices $B_k$ of $T$ in the unique path between $B_i$ and $B_j$ contain $v$ as well. 
\item For every edge $uv$ in $G$, there is $B_i$ that contains both $u$ and $v$. 
\end{enumerate}
The \emph{width} of a tree decomposition is the size of its largest set $B_i$ minus one. The \emph{treewidth} of a graph $G$, denoted $\tw(G)$, is the minimum width among all possible tree decompositions of $G$. 

The key property of bounded-treewidth graphs is that they have small ``balanced'' separators. 
To define a balanced separator, we generalize the notion of centrality (\ref{eq:cv}) for any set $S\subset V$ by writing
\begin{equation}\label{eq:cv2}
c(S)\;=\;\frac{1}{|V\setminus S|}\max_{C\in \cC^S}|C|~,
\end{equation}
where recall that $\cC^S$ is the set of connected components of the graph induced by $V\setminus S$.
We say that a separator $S$ is \emph{balanced} when $c(S)\leq 0.93$. We start by noting that every graph with bounded treewidth has a small balanced separator, see, for example, \cite[Theorem 19]{bodlaender1998partial}.
 \begin{proposition}\label{treewidth}
If $\tw(G)\leq k$ then $G$ has a separator $S$ such that $|S|\leq k+1$ and $c(S)\leq \frac{1}{2}\frac{|V|-k}{|V|-(k+1)}$. 
\end{proposition}
\begin{remark}
By \cite[Lemma~11]{bodlaender1998partial}, if $\tw(G)\leq k$ then $\tw(H)\leq k$ for every subgraph $H$ of $G$. In particular, one can recursively split a graph into subgraphs of small treewidth using small balanced separators. 	\end{remark}

\subsection{Finding a separator of two sets}\label{sec:ABsep}

Let $G=(V,E)$ be a graph and let $\Sigma\in \mathcal M(G)$ be generic. We give an algorithm that finds a minimal separator of two subsets $A,B\subset V$. By Assumption~\ref{mntool}, the size of such a minimal separator is $r:={\rm rank}(\Sigma_{A,B})$. Denote by $\cS(A,B)$ the set of all minimal separators of $A$ and $B$ in $G$. 
Denote by $U$ the set of all vertices that lie in some minimal separator in $\cS(A,B)$.

\begin{lemma}\label{lem:findesep}
	A vertex $v\in V$ lies in $U$ if and only if ${\rm rank}(\Sigma_{Av,Bv})=r$.  \end{lemma}
\begin{proof}
	This follows immediately from Lemma~\ref{mntool3}. \end{proof}

Lemma~\ref{lem:findesep} together with Lemma~\ref{mntool3} give a simple and efficient procedure to find an element in $\cS(A,B)$, detailed in Algorithm \ref{alg:sep0}.  
\begin{algorithm}
\label{alg:sep0}
$U \gets \emptyset$\;
$r = {\rm rank}(\Sigma_{A,B})$\;
\ForAll{$v\in V$}{\If{${\rm rank}(\Sigma_{Av,Bv})=r$}{$U\gets U\cup \{v\}$\;}}
$C\gets \{v_0\}$ for some $v_0\in U$\;
\ForAll{$u\in U\setminus \{v_0\}$}{\uIf{${\rm rank}(\Sigma_{ACu,BCu})=r$}{
    $C\gets C\cup \{u\}$ \;
  } }
    \Return{ $C$\;}
\caption{${\tt ABSeparator}(A,B)$}
\end{algorithm}
\begin{proposition}\label{prop:ABSep}Let $G=(V,E)$ be a graph and let $\Sigma\in \mathcal M(G)$ be generic. For any $A,B\subset V$ with $M=\max\{|A|,|B|\}$,  Algorithm~\ref{alg:sep0} finds a minimal separator of $A$ and $B$ with query complexity $\mathcal O(|V|M^2)$ and  computational complexity $\mathcal O(|V|M^3)$.
\end{proposition}

\begin{proof}
	By Lemma~\ref{lem:findesep}, the first loop finds the set $U$ of all vertices that lie in a minimal separator of $A$ and $B$. This loop has  $\mathcal O(|V|M^2)$ and $\mathcal O(|V|M^3)$ query and computational complexity, respectively. 
	
We now take an arbitrary vertex $v_0\in U$ and show that the second loop of the algorithm finds a minimal separator of $A$ and $B$ that contains $v_0$. Start with $C=\{v_0\}$ and note that, since $v_0\in U$, there exists $S\in \cS(A,B)$ containing $v_0$. In each iteration of the second loop we add  $u\in U\setminus \{v_0\}$ to $C$ if ${\rm rank}(\Sigma_{ACu,BCu})=r$. Therefore, by Lemma~\ref{mntool3}, we update $C\gets C\cup \{u\}$ if and only if there exists (not necessarily unique) $S\in \cS(A,B)$ such that $C\cup \{u\} \subseteq S$. By Assumption~\ref{mntool}, $|S|=r$ and so, if $|C\cup \{u\}|=r$ then $S=C\cup \{u\}$ and the rank condition will \emph{not} be satisfied for the subsequent vertices in the loop (showing correctness of the algorithm). If $|C\cup \{u\}|<r$ then $C\cup \{u\}$ is a strict subset of $S$ and all the remaining vertices in $S\setminus (C\cup \{u\})$ appear in the second loop \emph{after} $u$. Applying this argument recursively, we conclude correctness of the algorithm. 

Since $|U|\leq |V|$ and $r\leq M$, the number of queries and computational complexity of the second loop are $\mathcal O(|V|M^2)$ and $\mathcal O(|V|M^3)$, respectively, which concludes the proof.
	\end{proof}

\subsection{Balanced separators in $G$}\label{sec:vc}

In Section~\ref{sec:ABsep} we provided an efficient procedure that finds a separator for a given pair of sets $A,B\subset V$. In this section we show how to construct such a pair of small sets so that the obtained separator is, with high probability, a balanced separator for the entire graph $G$. 

Our approach to finding a balanced separator is to base the search on a random subset $W\subset V$ of size $m$ which can be handled within our computational budget. To argue why our randomization works and guide the choice of the parameter $m$ we rely on 
VC-theory initiated by \cite{VaCh71}. Let $\mathcal F_S$ be the set of all connected components in $\cC^S$ and their complements in $V\setminus S$. Write 
\begin{equation}\label{eq:Fk}
\mathcal F_k\;\;:=\;\;\bigcup_{S: |S|\leq k} \mathcal F_S~,	
\end{equation}
that is, $C\in \mathcal F_k$ if it is a connected component of $G\setminus S$ for some $S$ with at most $k$ elements, or $C$ is a union of all but one such components. 

\begin{definition}
A set $W \subseteq V$ is a \emph{$\delta$-sample} for $\mathcal F_k$ if for all sets $C\in \mathcal F_k$, 
\begin{equation}\label{eq:dsample}
\frac{|C|}{|V|}-\delta\;\leq\;\frac{|W\cap C|}{|W|}\;\leq\; \frac{|C|}{|V|}+\delta.  	
\end{equation}
\end{definition}

We now present conditions that assure that a uniformly random sample $W$ from the vertex set $V$ is a $\delta$-sample with high probability. A subset $W\subset V$ is \emph{shattered} by $\mathcal F_{k}$ if $ W\cap \mathcal F_{k}\;=\;\{W\cap C:\;C\in \mathcal F_{k}\}$ is the set of all subsets of $W$. Define the \emph{VC-dimension} of $\mathcal F_k$ , denoted by ${\rm VC}(\mathcal F_k)$, to be the maximal size of a subset shattered by $\mathcal F_{k}$. The following follows from the classical Vapnik-Chervonenkis inequality (see \cite{DeLu00} for 
a version that implies the constants shown here):
\begin{theorem}\label{netsample}
Suppose that ${\rm VC}(\mathcal F_k)= r$, $ \delta >0$, and $\tau \leq 1/2$. A set $W$ obtained by sampling $m$ vertices 
from $V$ uniformly at random, with replacement, 
 is a $\delta$-sample of $\mathcal F_{k}$ with probability at least $1-\tau$ if 
\begin{equation}\label{eq:boundW}
 m \;\geq\; \max\left(\frac{10r}{\delta ^2}\log\!\left(\frac{8r}{\delta ^2}\right), \frac{2}{\delta ^2}\log\!\left(\frac{2}{\tau}\right) \right)~.	
\end{equation}
\end{theorem}

A key property of the set $\mathcal{F}_k$ is that its VC-dimension is bounded by a linear function of the treewidth $k$. 
\begin{lemma}[\cite{feige2006finding}]\label{vc}
	Let $G=(V,E)$ be a graph and let $\mathcal F_k$ for $k\geq 1$ be the set defined in (\ref{eq:Fk}). Then ${\rm VC}(\mathcal F_{k})\leq 11\cdot k$. 
	\end{lemma}
	\begin{remark}
		The statement of this lemma in \cite{feige2006finding} uses a universal constant. Their proof however allows one to specify this constant to be $11$.
	\end{remark}
The next result shows that a $\delta$-sample admits a balanced separator. 

\begin{proposition}\label{prop:ifdsample} Let $G=(V,E)$ be such that $\tw(G)\leq k$, and let $W\subset V$ be a $\delta$-sample of $\mathcal F_{k+1}$ satisfying $|W|\geq 6(k+1)$. If $\delta\leq \frac{1}{24}$ then we can partition $W$ into two sets $A,B$ such that $|A|,|B|\leq \frac{2|W|}{3}$ and a minimal separator $S$ of $A$ and $B$ has at most $k+1$ elements. Moreover, for any such partition,
$\max(|A\setminus S|,|B\setminus S|) \leq \frac{4}{5}|W\setminus S |$.
\end{proposition}
In the latter proposition we refer to $\mathcal{F}_{k+1}$ and not $\mathcal{F}_k$, in order to be consistent with  Proposition~\ref{treewidth}.

\noindent Before we prove this result we formulate two useful lemmas.
The first one is merely a simple observation.
\begin{lemma}\label{lem:23partition}
	Let $U=\bigcup_{i=1}^d C_i$ be a partition of $U$ into disjoint sets such that $|C_1|\geq \cdots\geq |C_d|\geq 0$. If $|C_1|\leq \alpha |U|$ for $\alpha\geq \frac{2}{3}$ then $(1-\alpha)|U|\leq \sum_{i=1}^t |C_i| \leq \alpha|U|$ for some $t$.
\end{lemma}
To state the second lemma,
write $\lambda=1-\frac{k+1}{|V|}$ and note that $\frac{|V\setminus S|}{|V|}\geq \lambda$ for all $S$ such that $|S|\leq k+1$. 
\begin{lemma}\label{lem:bdsoncw}
	Suppose $W\subset V$ is a $\delta$-sample for $\mathcal F_{k+1}$ and let $S\subset V$ with $|S|\leq k+1$. If $C\in \cC^S$, then 
	\begin{equation}\label{eq:aux1}
		\frac{\lambda c(S)-\delta}{\lambda+2\delta}\;\;\leq \;\;\frac{|W\cap C|}{|W\setminus S|}\;\;\leq \;\;\frac{\lambda c(S)+\delta}{\lambda-2\delta}.
\end{equation}
\end{lemma}
\begin{proof}
	Using (\ref{eq:dsample}), we get
$$
\frac{|W\cap C|}{|W\setminus S|}\;\;\leq\;\; \left(\frac{|C|}{|V|}+\delta\right)\frac{|W|}{|W\setminus S|}\;\;\leq\;\;\left(c(S)\frac{|V\setminus S|}{|V|}+\delta\right)\frac{|W|}{|W\setminus S|}.
$$
To bound the last expression, let $\bar C=V\setminus (C\cup S)$. Since $C,\bar C\in \mathcal F_{k+1}$, we get
$$\frac{|W\setminus S|}{|W|}\;=\;\frac{|C\cap W|}{|W|}+\frac{|\bar C\cap W|}{|W|}\;\geq \;\left(\frac{|C|}{|V|}+\frac{|\bar C|}{|V|}-2\delta\right) \;=\; \frac{|V\setminus S|}{|V|}-2\delta.$$
A similar argument gives an upper bound for $\frac{|W\setminus S|}{|W|}$, which after taking reciprocals gives
\begin{equation}\label{eq:WS}
\frac{1}{\frac{|V\setminus S|}{|V|}+2\delta}\;\leq\;\frac{|W|}{|W\setminus S|}\;\leq\; \frac{1}{\frac{|V\setminus S|}{|V|}-2\delta}	
\end{equation}
This gives the upper bound in (\ref{eq:aux1}) because
$$
\frac{|W\cap C|}{|W\setminus S|}\;\;\leq\;\;\frac{c(S)\frac{|V\setminus S|}{|V|}+\delta}{\frac{|V\setminus S|}{|V|}-2\delta}\;\;\leq\;\;\frac{\lambda c(S)+\delta}{\lambda-2\delta},
$$
where the last inequality follows by the fact that the middle expression is a decreasing function of $\frac{|V\setminus S|}{|V|}$ and $\frac{|V\setminus S|}{|V|}\geq \lambda$. This establishes the upper bound in (\ref{eq:aux1}). The lower bound follows by similar arguments. 
\end{proof}

\begin{proof}[Proof of Proposition~\ref{prop:ifdsample}]
Let $S^*$ be a minimizer of $c(S)$  among all $S\subset V$ such that $|S|\leq k+1$. By Proposition~\ref{treewidth} $c(S^*)\leq\frac{1}{2}\frac{|V|-k}{|V|-(k+1)}$, which is further bounded by $11/20$ if $|V|\geq 6(k+1)$. By Lemma~\ref{lem:bdsoncw}, if  $C\in \cC^{S^*}$ then
$$
\frac{|W\cap C|}{|W\setminus S^*|}\;\leq\;\frac{\frac{11}{20}\lambda+\delta}{\lambda-2\delta}. 
$$
The right-hand side is an increasing function of $\delta$ and the maximum for $\delta\leq \frac{1}{24}$ is $(\frac{11}{20}\lambda+\frac{1}{24})/(\lambda-\frac{1}{12})$, which is bounded by $\frac{2}{3}$ because $\lambda\geq \frac{5}{6}$ (use $|V|\geq |W|\geq 6(k+1)$). This shows that $W\setminus S^*$ can be partitioned into disjoint subsets $W\cap C$ for $C\in \cC^{S^*}$ all of size at most $\frac{2}{3}|W\setminus S^*|$. By Lemma~\ref{lem:23partition}, we can group these sets into two groups $A'$ and $B'$ each of size at most $\frac{2}{3}|W\setminus S^*|$. To show the first claim let $A,B$ be any two sets partitioning $W$ that satisfy $A\setminus S^*=A'$, $B\setminus S^*=B'$. We next show that there is a choice of $A,B$ that gives $\max(|A|,|B|)\leq \frac{2}{3}|W|$. This is done by allocating the elements of $W\cap S^*$ in a balanced way between $A'$ and $B'$ so that both $A'$ and $B'$ get at most $\frac{2}{3}$ of the elements in $W\cap S^*$. This can be always done if $W\cap S^*$ has at least two elements. If $W\cap S^*$ is empty, the statement is trivial. If $|W\cap S^*|=1$ we consider two cases (i) $|A'|<|B'|$ and (ii) $|A'|=|B'|$. In case (i) we allocate the element in $W\cap S^*$ to $A'$. In that case
$$
|A|\;=\;|A'|+1\;\leq\; |B'|\;=\;|B|\;\leq \;\frac{2}{3}|W\setminus S^*|\;\leq\; \frac{2}{3}|W|.
$$
In case (ii), we again allocate the element in $W\cap S^*$ to $A'$ and use the fact that $|A'|=|B'|=\frac{1}{2}|W\setminus S^*|$, which gives
$$
|B|\;\leq\; |A|\;=\;|A'|+1\;=\;\frac{1}{2}|W\setminus S^*|+1\;=\;\frac{1}{2}|W|+\frac{1}{2}\;\leq\; \frac{2}{3}|W|,
$$
where the last inequality holds always if $|W|\geq 3$. This proves the first claim.

To show the second claim, assume $\max(|A|,|B|) \leq \frac{2}{3}|W|$. Since $|W|\geq 6(k+1)$ we have $|W\setminus S|\geq 5(k+1)\geq 5|W\cap S|$. Now
$$
\max(|A\setminus S|,|B\setminus S|)\leq \frac{2}{3}|W|\leq \frac{2}{3}(|W\setminus S|+|W\cap S|)\;\leq \; \frac{2}{3}\left(1+\frac{1}{5}\right)|W\setminus S|,
$$
which completes the argument.\end{proof}

\subsection{Separating and splitting}

We now propose a procedure ${\tt Separator}$ that finds a balanced separator in $G$. When a separator is found, decomposing the graph into connected components is straightforward and is given in the procedure ${\tt Components}$. 

The procedure starts by choosing a sample $W\subset V$. Then the algorithm looks for a partition of $W$ into two sets $A,B$ so that $|A|,|B|\leq \frac{2}{3}|W|$ and the rank of $\Sigma_{A,B}$ is small. In Proposition~\ref{prop:dsamplesep} we argue why such a partition exists with high probability. Then the algorithm uses ideas of Section~\ref{sec:ABsep} to efficiently find a minimal separator $S$ of $A$ and $B$ in $G$; by construction $|S|={\rm rank}(\Sigma_{A,B})$. At this moment a purely deterministic part of the process begins. Given $S$, the algorithm decomposes $V$ into connected components in $\cC^S$. This is done using rank conditions like in the tree-like case.

\begin{algorithm}
\label{alg:separator}
Pick a set $W$ by taking $m$ vertices uniformly at random, where $m$ satisfies (\ref{eq:boundW}) with $r=11k$ and $\delta=1/24$\;
Search exhaustively through all partitions of $W$ into sets
$ A,B$ with $|A|,|B|\leq \frac{2}{3}|W|$, minimizing ${\rm rank}(\Sigma_{A,B})$\;
If no balanced split exists, output any partition $ A,B$ of $W$\;
$S\gets {\tt ABSeparator}(A,B)$\;
\Return{ $S$}
  \label{alg:sep}
\caption{${\tt Separator}$}
\end{algorithm}

\begin{algorithm}
\label{alg:components}
  $S\leftarrow {\tt Separator(V)}$\; 
  $r\gets |S|$\;  
  $R \gets \emptyset $\tcp*{will contain one vertex from each $C\in \cC^S$}
  \For{$v\in V\setminus S$}{
  $notFound \gets True$\;

  \For{$u\in R$}{
    
  	\If{${\rm rank}(\Sigma_{uS,vS}) = r+1$ }{$C_u\leftarrow C_u\cup\{v\}$\; $notFound \gets False$\;}}
  	\If{$notFound$}{create $C_{v}= \left\{v\right\}$\;$R\leftarrow R\cup\{v\}$\;} }
  	\Return{  \upshape$S$ and all $C_v$ for $v\in R$\;}
\label{tlsplit2}
\caption{${\tt Components(V)}$}
\end{algorithm}

The next proposition shows that Algorithm~\ref{tlsplit2} outputs a balanced separator with high probability.

\begin{proposition}\label{prop:dsamplesep}
Let $G=(V,E)$ be a graph with ${\rm tw}(G)\leq k$, and $|V|\geq 6(k+1)$ vertices and let $\Sigma\in \mathcal M(G)$ be generic.
Let $\tau\in (0,1)$.  Then, with probability at least $1-\tau$, Algorithm~\ref{alg:sep} finds a separator $S$ in $G$ such that $|S|\leq k+1$ and $|C|\leq 0.93|V|$ for each connected component $C\in \cC^S$. 
	\end{proposition}

\begin{proof}
	The procedure starts by choosing a sample $W\subset V$. The size of the sample $m$ is chosen so that, with probability at least $1-\tau$, $W$ is a $\delta$-sample for $\mathcal F_{k+1}$. A sufficient condition for $m$ follows by Theorem~\ref{netsample} and Lemma~\ref{vc}. Note that this condition also assures that $|W|\geq 6(k+1)$. 
	
	Since $\delta=1/24$, by Proposition~\ref{prop:ifdsample}, we can partition $W$ into two sets $A,B$ such that $|A|,|B|\leq \frac{2|W|}{3}$ and any minimal separator $S$ of $A$ and $B$ has at most $k+1$ elements. Moreover, for any such partition $|A\setminus S|,|B\setminus S|\leq \frac{4|W\setminus S |}{5}$. Now we only need to show that for each connected component $C\in \cC^S$ of $G$, $|C|\leq \frac{93}{100}|V|$. Indeed, if $C^*$ is the maximal component in $\cC^{S}$ then $\frac{|W\cap C^*|}{|W\setminus S|}\leq \frac{4}{5}$. Since $C^*$ 	lies in $\mathcal F_{k+1}$, we get
$$
\frac{|C^*\cap W|}{|W\setminus S|}\;\geq\;\left(\frac{|C^*|}{|V\setminus S|}\frac{|V\setminus S|}{|V|}-\delta\right)\frac{|W|}{|W\setminus S|}\overset{(\ref{eq:WS})}{\geq}\frac{c(S)\frac{|V\setminus S|}{|V|}-\delta}{\frac{|V\setminus S|}{|V|}+2\delta}.
$$
The expression on the right-hand side is an increasing function of $\frac{|V\setminus S|}{|V|}$ and $\frac{|V\setminus S|}{|V|}\geq \lambda$, which gives that
$$
\frac{|C^*\cap W|}{|W\setminus S|}\;\geq\;\frac{\lambda\frac{|C^*|}{|V\setminus S|}-\delta}{\lambda+2\delta}
$$
Since $\frac{|C^*\cap W|}{|W\setminus S|}\leq \frac{4}{5}$, $\delta\leq \frac{1}{24}$, and $\lambda\geq \frac{5}{6}$, we get that $\frac{|C^*|}{|V\setminus S|}\leq \frac{93}{100}$.

\end{proof}

\section{Recovery of bounded treewidth graphs}\label{sec:boundedTW}

In this section we present an algorithm for reconstructing graphs with bounded treewidth. Let $G=([n],E)$ be a graph and let $\Sigma\in \mathcal M(G)$ be generic. To recover $G$ from $\Sigma$ we follow a similar divide-and-conquer strategy as in the previous sections. First, a balanced separator $S$ is chosen and then the algorithm recurses into all the components in $\cC^S$. There is however a complication. If $C\in \cC^S$, then $G[C\cup S]$ is not equal to the graph of $\Sigma_{CS}$, unless for \emph{each} $C\in \cC^S$ the set of vertices in $S$ that are linked to $C$ by an edge is a clique of $G$; see \cite[Theorem 3.3]{frydenberg1990marginalization}. The condition holds, in particular, when $S$ is a clique, but in our case there is no way to assure that in general. If it does not hold, then the graph of $\mathcal G(\Sigma_{CS})$  is strictly bigger than the subgraph $G[C\cup S]$.
The next example illustrates this phenomenon.

\begin{example}
Consider the four-cycle given below together with the corresponding covariance and precision matrices
$$
\begin{tikzpicture}[baseline=-0.8cm,scale=0.6, every node/.style={scale=0.6},node distance=2.5cm,
        main node/.style={circle,draw,minimum size=1cm,inner sep=0pt]}]
    \node[main node] (1) {$1$};
    \node[main node] (2) [right of=1]  {$2$};
    \node[main node] (3) [below of=2] {$3$};
    \node[main node] (4) [left of=3] {$4$};
    \draw (1) -- (2) -- (3) -- (4) -- (1);
    \end{tikzpicture}\qquad\qquad
\Sigma\;=\;\begin{bmatrix}
	7 & -2 & 1 & -2\\
		-2 & 7 & -2 & 1\\
	1 & -2 & 7 & -2\\
	-2 & 1 & -2 & 7
\end{bmatrix}\qquad\qquad K=\frac{1}{24}\begin{bmatrix}
	4 & 1 & 0 & 1\\
		1 & 4 & 1 & 0\\
	0 & 1 & 4 & 1\\
	1 & 0 & 1 & 4
\end{bmatrix}
$$ 
Then $\{1,3\}$ separates $2$ and $4$ but the graph $\mathcal G(\Sigma_{123})$ is the complete graph over $\{1,2,3\}$ because
$$
  	\begin{bmatrix}
	7 & -2 & 1 \\
		-2 & 7 & -2 \\
	1 & -2 & 7 \\
\end{bmatrix}^{-1}\;=\;\frac{1}{96}\begin{bmatrix}
	15 & 4 & -1 \\
		4 & 16 & 4 \\
	-1 & 4 & 15 
\end{bmatrix}.$$
\end{example}

An easy way around this problem is by noting that Gaussian graphical models are closed under conditioning and this probabilistic statement has a useful algebraic counterpart. The graph of the conditional covariance is obtained from $G$ by removing the vertices in the conditioning set and all the incident edges. This means that  the edges in $G[C]$ can be recovered from the conditional covariance matrix 
\begin{equation}\label{eq:condcov}
\Sigma_{C|S}\; :=\; \Sigma_{C,C}-\Sigma_{C,S}\Sigma_{S,S}^{-1}\Sigma_{S,C}.	
\end{equation}
More concretely, we have the following basic result.

\begin{lemma}\label{lem:laststep}
	If $S$ separates $C$ from the rest of $G=\mathcal G(\Sigma)$ then $K_{C,C}=(\Sigma_{C|S})^{-1}$.
\end{lemma}
This result can be argued by standard properties of the Gaussian distribution as $K_{C}$ is the inverse of the conditional covariance matrix $\Sigma_{C|[n]\setminus C}$; see, for example, equation (C.3) in \cite{lauritzen1996graphical}. If $S$ separates $C$ from the remaining vertices then, by conditional independence, this conditional covariance matrix is equal to $\Sigma_{C|S}$. For clarity and completeness, we also include an independent purely algebraic argument. 
\begin{proof}First note that, if $S$ separates $C$ from the rest $B=V\setminus (C\cup S)$ we have
$$
\Sigma_{C,B}\;=\;\Sigma_{C,S}\Sigma_{S,S}^{-1}\Sigma_{S,B}.
$$ 
Indeed, by Lemma~\ref{mntool3}, ${\rm rank}(\Sigma_{CS,SB})=|S|$. Using the Guttman rank additivity formula given below in (\ref{eq:guttman}) we conclude that the matrix $\Sigma_{C,B|S}:=\Sigma_{C,B}-\Sigma_{C,S}\Sigma_{S,S}^{-1}\Sigma_{S,B}$ is zero. Moreover, the matrix equation 
\begin{equation}\label{eq:bigmatrprod}
\begin{bmatrix}
\Sigma_{C,C} & \Sigma_{C,S} & \Sigma_{C,S}\Sigma_{S,S}^{-1}\Sigma_{S,B}\\	
\Sigma_{C,C} & \Sigma_{C,S} & \Sigma_{C,B}\\	
\Sigma_{B,S}\Sigma_{S,S}^{-1}\Sigma_{S,C} & \Sigma_{C,S} & \Sigma_{C,B}\\	
\end{bmatrix}\cdot \begin{bmatrix}
\mathbf A & \mathbf B & \mathbf 0\\	
\mathbf B^T & \mathbf C & \mathbf D\\	
\mathbf 0 & \mathbf D^T & \mathbf E\\	
\end{bmatrix}\;=\;I_n	
\end{equation}
is satisfied for 
$$
\mathbf A=(\Sigma_{C|S})^{-1},\quad \mathbf B=-(\Sigma_{C|S})^{-1}\Sigma_{C,S}\Sigma_{S,S}^{-1},
$$
$$
\mathbf E=(\Sigma_{B|S})^{-1},\quad \mathbf D=-(\Sigma_{B|S})^{-1}\Sigma_{B,S}\Sigma_{S,S}^{-1},
$$
and
$$
\mathbf C=\Sigma_{S,S}^{-1}\left(\Sigma_{S,S}+\Sigma_{S,C}(\Sigma_{C|S})^{-1}\Sigma_{C,S}+\Sigma_{S,B}(\Sigma_{B|S})^{-1}\Sigma_{B,S}\right)\Sigma_{S,S}^{-1}.
$$
The first term in the product in (\ref{eq:bigmatrprod}) is $\Sigma$. Thus the second term is $K=\Sigma^{-1}$. In particular, $K_{C,C}=\mathbf A=(\Sigma_{C|S})^{-1}$.
\end{proof}

This result shows that in order to keep information about the induced subgraph $G[C]$ once we regress on $C$, it is important to keep the information about the separating set. To see how this is done, it is helpful to study the situation in Figure~\ref{fig:recursion}. Suppose that $S$ separates $G$ into several components one of which is $C$. We then recurse our algorithm on $C$ by conditioning on $S$. In the next step we use the matrix $\Sigma_{C|S}$ to find a balanced separator $S'$ of $G[C]$. We then recurse on the corresponding components $C_1,C_2,C_3,C_4$. Note that in the next step it is not enough to condition on $S'$ to study $G[C_2]$ because it is connected to the rest of the graph through $S$. Therefore, in this recursive call we need to work with the conditional covariance matrix $\Sigma_{C_2|SS'}$. 

The dependence on separating sets requires a modification of the algorithms that we use to decompose the graph. Instead of working on the covariance matrix, they should be working on the conditional covariance matrix. Note however, that rank queries for $\Sigma_{A,B|S}$ with $A,B\subset C$ are equivalent to rank queries on  $\Sigma_{AS,BS}$. Indeed,  by the Guttman rank additivity formula (see e.g. \cite[Section 0.9]{zhang})
\begin{equation}\label{eq:guttman}
{\rm rank}(\Sigma_{AS,BS})\;=\;{\rm rank}(\Sigma_{S})+{\rm rank}(\Sigma_{A,B|S})\;=\;|S|+{\rm rank}(\Sigma_{A,B|S}).	
\end{equation}
Therefore the algorithms ${\tt ABSeparator}(A,B)$, ${\tt Separator}$, ${\tt Components}$ have their simple modifications  ${\tt ABSeparator}(V,A,B,S)$, ${\tt Separator}(V,S)$, ${\tt Components}(V,S)$, where a set $S$  disjoint from $V$ is added to both the row and the column set in all the rank queries. For completeness we explicitly provide these algorithms in Appendix~\ref{app:algs}.

In Algorithm~\ref{main} we present the complete algorithm which relies on routine ${\tt Reconstruct}$, which is then called recursively in Algorithm~\ref{alg-learn-struct}. With a fixed bound $k$ on the treewidth of $G$, the main algorithm returns the precision matrix $K=\Sigma^{-1}$.

\begin{algorithm}
{$ \wh K\gets 0\in \R^{n\times n}$    \\
 Fix $m$ satisfying (\ref{eq:boundW}) with $r=11k$ and $ \delta=\frac{1}{24}$\\
Reconstruct${\tt([n],\emptyset)}$
\label{main}
\caption{$Main\;algorithm$}}
\end{algorithm}

At each call of {\tt Reconstruct}$(V,S)$, if the input vertex set $V$ is larger than the fixed threshold $m$, then ${\tt Separator}(V)$ finds a balanced separator $S'$ of $G[V]$. Then the procedure ${\tt Components}$ finds all connected components $C_i$ in $G[V]\setminus S'$. Subsequently, ${\tt Reconstruct}$ recurses in all these components replacing $S$ with $S\cup S'$ as in  Figure~\ref{fig:recursion}.

\begin{figure}\centering
	\includegraphics[scale=.7]{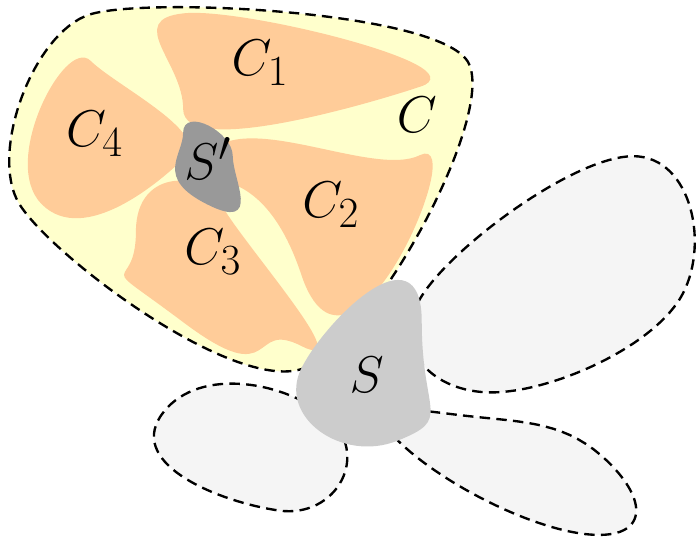}
	\caption{Components found during subsequent recursive calls to {\tt Reconstruct}. The algorithm needs to keep track of the separators found at all levels, as $C_2$ and $C_3$ are connected by both $S'$ and $S$.}\label{fig:recursion}
\end{figure}

Most edges in $K$ are only reconstructed in the final recursive calls. Consider the situation in Figure~\ref{fig:recursion}. Suppose that ${\tt Reconstruct}(C,S)$ recurses to ${\tt Reconstruct}(C_1,S\cup S')$. If $|C_1|\leq m$ then ${\tt Reconstruct}(C_1,S\cup S')$ computes $K_{C_1}$, which by Lemma~\ref{lem:laststep} is equal to the inverse of $\Sigma_{C_1|SS'}$. The matrices $K_{C_1 S'}$ and $K_{S'S'}$ can be computed in a similar way as described in the lemma below. \begin{lemma}\label{lem:recurse}
Suppose $C\in \cC^S$ and that $C$ is further decomposed into $S'$ and the connected components $\{C_1,\ldots,C_d\}$ 
(as in Figure~\ref{fig:recursion}). Let $K=\Sigma^{-1}$. The submatrix $K_C$ has a block structure with $K_{C_i,C_j}=0$ for $i\neq j$ and
$$
   K_{C_i,S'} = -K_{C_i}\Sigma_{C_i,S'|S}\Sigma_{S'|S}^{-1},\qquad 
K_{S'} = \left(\mathbb I_{|S'|}-\sum_{i=1}^d K_{S',C_i}\Sigma_{C_i,S'|S}\right)\Sigma_{S'|S}^{-1},
$$
where $\mathbb I_m$ denotes the $m\times m$ identity matrix. 
\end{lemma}
\begin{proof}
There are no direct links between $C_i$ and $C_j$ in $G$ and so  $K_{C_i,C_j}=0$ for $i\neq j$. Lemma~\ref{lem:laststep} gives the identity $K_C \Sigma_{C|S}=\mathbb I_{|C|}$. Taking the $C_i$-rows of $K_C$ and the $S'$-columns of $\Sigma_{C|S}$ we get from this identity that
$$
K_{C_i}\Sigma_{C_i,S'|S}+K_{C_i,S'}\Sigma_{S',S'|S}\;=\;0, 
$$
which implies the first equality. Taking the $S'$-rows and $S'$-columns we get
$$
\sum_{i=1}^d K_{S',C_i}\Sigma_{C_i,S'|S}+K_{S'}\Sigma_{S'|S}\;=\;\mathbb I_{|S'|},
$$
which implies the second formula. 
\end{proof}

\begin{algorithm}
 \uIf{$|V|> m$}
 {$C_{1},\dots,C_{d},S'\leftarrow {\tt Components}(V,S)$\;
 \For{$i$ \upshape from $1$ to $d$}{ 
		$\wh K_{C_i}\gets{\tt Reconstruct}(C_{i},S\cup S')$\;
		$\wh K_{C_i,S'}\gets - \wh K_{C_i}\Sigma_{C_i,S'|S}\Sigma_{S'|S}^{-1}$	
		} 
		$\wh K_{S'}\gets \left(\mathbb I_{|S'|}-\sum_{i=1}^d \wh K_{S',C_i}\Sigma_{C_i,S'|S}\right)\Sigma_{S'|S}^{-1}$\;}
\Else{\Return $\wh K_V$} 
\label{alg-learn-struct}
\caption{${\tt Reconstruct}(V,S)$}
\end{algorithm}

\begin{theorem}
${\tt Reconstruct}([n],\emptyset)$ correctly recovers $K=\Sigma^{-1}$. 
\end{theorem}
\begin{proof}
The correctness of {\tt Components}$([n],\emptyset)$ was already shown in Proposition~\ref{prop:dsamplesep}. Using the discussion given in the beginning of this section, we can easily adjust this proof for any call of {\tt Components}$(V,S)$. This, together with Lemma~\ref{lem:laststep}, implies that in each call of {\tt Components}$(V,S)$:
\begin{enumerate}\item[(i)]  If $u,v$ are put in different components components then $K_{uv} =0$.
\item[(ii)]  It holds that $K_{C_i}=(\Sigma _{C_i|SS'})^{-1}$.
\end{enumerate}
By (i) all (zero) entries in $K_{C_i,C_j}$ are correctly recovered. By (ii) all the entries of $K_{C_i}$ obtained by inverting $\Sigma _{C_i|SS'}$ are correct. By Lemma~\ref{lem:recurse}, also $K_{C_i,S'}$ and $K_{S',S'}$ are correctly recovered.  \end{proof}

\begin{theorem}\label{th:mainmain}
Let $G=([n],E)$ be a graph with treewidth ${\rm tw}(G)\leq k$ and 
maximum degree $\Delta(G)\le d$. Let $\Sigma\in \mathcal M(G)$ be generic and let $m$ in ${\tt Separator}$ satisfy (\ref{eq:boundW}) with $ r\geq 11k$, $\delta\leq \frac{1}{24}$, and $\tau \leq \frac{1}{3}$.
Then, with probability at least $1-\frac{1}{n^8}$, the query complexity of {\tt Reconstruct} is of the order 
\[
\mathcal O\left(
(2^{\mathcal{O}(k\log k)}+
dk\log n)k^2n\log ^3 n
\right),
\]
and the time complexity is of the order
\[
\mathcal O\left(
(2^{\mathcal{O}(k\log k)}+
dk\log n)k^3n\log ^4 n 
\right).
\] 
\end{theorem}

Before we prove this result, a number of remarks are in order.

\begin{remark}
 The choice of $\tau\leq 1/3$ in Theorem~\ref{th:mainmain} is arbitrary; any other choice would only change the constant factors in the complexity bounds and the probability of exceeding the stated query and time complexity. Moreover, we note that the total probability of error can be made arbitrarily small. In the following proof we show that the recursion depth is $\mathcal{O}\left(\log n\right)$ with high probability. Then, executing the algorithm 
$\mathcal{O}(\log \frac{1}{\epsilon})$ 
times (each time stopping if it does not finish in the time limit stated in Theorem~\ref{th:mainmain}) we get at least one timely execution of the algorithm with probability $1-\epsilon$, regardless of $n$. The complexity only changes by a factor of 
$\mathcal{O}(\log\frac{1}{\epsilon})$. \end{remark}

\begin{remark}
As it is seen from the proof, the assumption of bounded degree can be relaxed. It can be substituted 
by the assumption that removing $\mathcal O(k\log n)$ vertices decomposes the graph into at most 
a polylogarithmic number of connected components. In other words, one may weaken the bounded-degree assumption
by suitable assumptions on the \emph{fragmentation} of the graph. We refer to \cite{hajiaghayi2003note} and \cite {hajiaghayi2007subgraph}
for more information on this notion. An example of a graph with unbounded degree for which our reconstruction method works
is the wheel graph (i.e.,  the graph formed by connecting a central vertex to all vertices of a cycle of $n-1$ vertices). This
graph has treewidth $3$, maximum degree $n-1$ but low fragmentation. 
\end{remark}

\begin{remark}
The problem of computing the treewidth of a graph given its adjacency matrix is NP-hard~\cite{arnborg1987complexity} (see~\cite{bodlaender2016c} for an account on the history of the problem and references on current results). Hence it is not surprising to have an exponential dependence on treewidth. Note however that in our case we do not have access to edge queries but to separation queries, in contrast to the traditional setting. We are not aware of algorithmic results in this setting. We expect that similar hardness results should hold but proving such hardness seems non-trivial and outside of the scope of this paper.
\end{remark}

\begin{proof}[Proof of Theorem~\ref{th:mainmain}]
We refer to all operations in ${\tt Reconstruct}([n],\emptyset)$ excluding operations in subsequent calls ${\tt Reconstruct}(C_i,S')$ as the zeroth recursion level of the algorithm. Similarly, the operations of all ${\tt Reconstruct}(C_i,S')$ for $i=1,\ldots,d$ apart from their subsequent calls are called the first recursion level. We extend this definition iteratively to the $t$-th recursion level for $t>1$. 


Assume initially that ${\tt Components}$ never fails, that is, $|S'|\leq k+1$ and for each connected component $C\in \cC^S$ it holds that $|C|\leq 0.93|V|$, as stated in Proposition~\ref{prop:dsamplesep}. We will bound the total probability of failure later in the proof. Since the algorithm recurses on sets $C_i$ of size at most $0.93|V|$, the recursion depth (maximal number of recursion levels) is at most $\log_{100/93} n$, which is of order $\mathcal O(\log(n))$. Moreover, in each call of ${\tt Reconstruct}(V,S)$ always $|S|=\mathcal{O}(k\log n)$.

We start with the analysis of ${\tt Components}(V,S)$. Assume first that $|V|>m$ and write $s=|S|$ and $s'=|S'|$. Finding a balanced partition $A,B$ of $W$ in ${\tt Separator}$ is achieved by exhaustively searching all $<2^{m}$ balanced partitions and computing the rank of the associated matrices $\Sigma_{ASS',BSS'}$, which gives query complexity $\mathcal O (2^m(m+s)^2)$ and time complexity $\mathcal O(2^{m}(m+s)^3)$ for this step. Taking into account that $m=\mathcal O(k\log k)$ and $s=\mathcal{O}(k\log n)$, we 
obtain complexity  
$$\mathcal O (2^ms^2)\; \;\mathrm{and} \;\;\mathcal O(2^{m}s^3)$$
for queries and time, respectively. Then, given a balanced split, ${\tt ABSeparator}(V,A,B,S)$ finds a separator of $A$ and $B$ in $\mathcal O(|V|(m+s)^2)$ queries and $\mathcal O(|V|(m+s)^3)$ time. This bound can be obtained by a simple modification of  Proposition~\ref{prop:ABSep}, which gives bounds for ${\tt ABSeparator}(A,B)$. Hence we obtain complexity  
$$\mathcal O (|V|s^2)\;\; \;\mathrm{and} \;\;\;\mathcal O(|V|s^3)$$ 
for queries and time, respectively. 

Since removing at most  $s'$ vertices of degree at most $d$ splits the graph into at most $ds'$ connected components,
splitting $V$ into the connected components $C_i$ requires $\mathcal O(|V|s'ds)$ queries and $\mathcal O(|V|s'ds^3 )$ time, 
and since $s=\mathcal O(k\log n)$,
this gives query and time complexity bounds
$$\mathcal O (|V|dk^2\log n)\; \;\mathrm{and} \;\;\mathcal O(|V|dk^4\log^3 n).$$ 
In the case where $|V|\leq m$ , we obtain queries and time of the order $\mathcal O (s^2)$ and $\mathcal O(s^3)$ respectively. These terms are dominated by the terms that appeared for earlier steps of the algorithm and will be ignored in what follows. 
Overall, \emph{at each recursion level} this part of the algorithm requires 
\begin{equation*}\mathcal O\left(n2^mk^2\log^2n+nk^2\log^2n+ndk^2\log n\right)\end{equation*}
queries and
\begin{equation*}\mathcal O(n2^mk^3\log^3n+nk^3\log^3n+ndk^4\log^3n)
\end{equation*}
time, since the vertex sets $V$ are disjoint and there can be up to $O(n)$ calls to ${\tt Separator}$ at the bottom levels of recursion. Using the fact that $m=\mathcal O(k\log k)$ and simplifying, we obtain:
\begin{equation}\label{query1}\mathcal O\left(2^{\mathcal{O}(k\log k)}k^2n\log ^2 n+dk^2n\log^2 n\right)\end{equation}
queries and
\begin{equation}\label{time1}\mathcal O(2^{\mathcal{O}(k\log k)}k^3n\log ^3 n+dk^4n\log ^3 n )\end{equation}time. 

After calling ${\tt Components}(V,S)$, ${\tt Reconstruct}$ obtains $\wh K_{C_i}$ (we focus on a fixed recursion level so we can ignore the recursive call of ${\tt Reconstruct}(V,S)$) and computes the matrices $\wh K_{C_i,S'}$. After these matrices are computed for all components $C_i$, the algorithm computes  $\wh K_{S',S'}$. For these computations we need to calculate the conditional covariance matrices $\Sigma_{C_i,S'|S}$ and $\Sigma_{S'|S}$. The time to compute each $\Sigma_{C_i,S'|S}=\Sigma_{C_i,S'}-\Sigma_{C_i,S}\Sigma_{S}^{-1}\Sigma_{S,S'}$ requires 
$\mathcal O(|C_i|s+s^2)$ and $\mathcal O(|C_i|s^2+s^3)$ queries and time, respectively, hence 
$$\mathcal O(|V|s+s^2dk\log n)\;\;\mathrm{and} \;\;\mathcal O(|V|s^2+s^3dk\log n),$$ for all $C_i$.
 Computing $\wh K_{C_i,S'}=\wh K_{C_i}\Sigma_{C_i,S'|S}\Sigma_{S'}^{-1}$ requires only $\mathcal O(k^2)$ additional queries (we use $s'\leq k+1$). The time complexity of this computation is dominated by the time needed to compute $\wh K_{C_i}\Sigma_{C_i,S'|S}$. A naive method of computing $\wh K_{C_i}\Sigma_{C_i,S'|S}$ would require $\mathcal O(|C_i|^2k)$ time, which is too time-consuming for our purposes. However, we can take advantage of the fact that the subgraph $G[C_i]$ has treewidth at most $k$ (by the remark following Proposition~\ref{treewidth}) and the number of edges in such a graph is at most $|C_i|k$.
This implies that there are at most $k|C_i|$ non-zero entries in $\wh K_{C_i}$ and so the multiplication takes time $\mathcal{O}( k^2|C_i|)$. Considering all $C_i$, we obtain time complexity
$$\mathcal{O}( k^2|V|)$$
for this step.
 Finally, to compute $\wh K_{S'}$ we need additional $\mathcal O(sk)$ queries for $\Sigma_{S,S'}$ and $\mathcal O(|V|k^2+dk^3)$ time to compute $\wh K_{S'}$. However, since $s=\mathcal O(k\log n)$, these terms are clearly dominated by complexity of the preceding steps in the algorithm and so they will be ignored in what follows.  Overall, at each recursion level this part of the algorithm requires 
\begin{equation*}\mathcal O\left( ns+ns^2dk\log n\right)\qquad\mbox{and}\qquad \mathcal O(ns^2+ns^3dk\log n+k^2n)\end{equation*}
queries and time respectively. Using the fact that $m=\mathcal O(k\log k)$ and $s=\mathcal O(k\log n)$ and simplifying, we obtain: 
\begin{equation}\label{query2}\mathcal O\left(kn\log n+dk^3n\log ^3 n\right)\end{equation}
queries and
\begin{equation}\label{time2}\mathcal O(k^2n\log ^2n +dk^4n\log ^4n )\end{equation}time. 
 
The total complexity for a fixed recursive level are obtained by combining (\ref{query1})--(\ref{time2}). Taking into account the recursion depth $\mathcal O(\log(n))$ we get the stated overall complexity bounds. 

Let $I_i$ be the indicator variable that the $i$-th call of ${\tt Components}$ succeeds. This happens with $\PROB(I_i=1)=1-\tau\geq 2/3$. Let $\alpha =0.93$.  Consider a given recursion path from the root to a leaf in the recursion tree. There are at most $\log_{1/\alpha} n$ calls with $I_i=1$ in such a recursion path. Since the $I_i$ are independent, we can use Hoeffding's inequality to bound the probability that we have less than $\log _{1/\alpha}n$ successes in $N=3\log _{1/\alpha}n$ calls: 
$$
\PROB\left(\sum_{i=1}^N I_i<\frac{1}{3} N\right)\;\leq\;e^{-\frac{2}{9}N}\;=\;n^{-\frac{2}{3\log(1/\alpha)}}\;\leq\; n^{-9}.
$$
This argument implies that a fixed path from the root to a leaf in the recursion tree is logarithmic with high probability. There are at most $n$ such paths. Hence, by the union bound and the condition for $n$, the probability that there exists one of them with more than logarithmic length is bounded by $1/n^{8}$. 
\end{proof}

\section{Tree reconstruction using noisy covariance oracles}
\label{sec:finite}

In this paper we focus on the noiseless setting when entries of a
matrix $\Sigma\in \mathcal M(G)$ may be accessed by a learner and
these values are available exactly, without error -- this is our
``covariance oracle.'' In other words, we have shown that in many
cases it is possible to invert the positive definite matrix $\Sigma$ after seeing a
tiny, adaptively chosen, fraction of its entries.

In many learning problems, the entries of $\Sigma$ are not available
exactly. This is the case in statistical problems when $\Sigma$ is an
unknown covariance matrix of a random vector $X$ and its entries may
be estimated from data. In this section we discuss how the results of
this paper may be extended to situations when covariances may not be
observed exactly, for example, due to statistical fluctuations.

Here we limit ourselves to the study of the case when the
underlying graph $G$ is a tree. We show how Algorithm~\ref{tmain} may
be modified to handle noise and establish sufficient conditions that
guarantee correct recovery. Along similar lines, one may also modify the other algorithms
introduced in this paper (for recovery of tree-like graphs and graphs
of bounded treewidth). However, the details are somewhat 
tedious and go beyond the scope of this paper.

In order to simplify the presentation, we assume that all diagonal
elements of $\Sigma$ are equal to $1$, that is, $\Sigma$ is a
correlation matrix with entries $\sigma_{ij}$ for $i\neq j$.
The extension of the general case is straightforward, at the price of
visually more complicated formulas.

In the discussion that follows, we assume that the noisy covariance
oracle, when queried for the $(i,j)$-th entry of $\Sigma$, returns a
value $\wh{\sigma}_{ij}\in [-1,1]$ satisfying 
\begin{equation}\label{eq:4002}
\max _{ij}\left|\wh{\sigma}_{ij}-\sigma_{ij}\right|<\epsilon
\end{equation}
for some $\epsilon \in (0,1)$. We assume that $\wh\sigma_{ii}=1$ for all $i$. In a statistical setting when $\Sigma$ is the covariance matrix of a
random vector $X$ and one may obtain independent samples of $X$, it is
easy to construct such a noisy covariance oracle. We discuss this in more detail
at the end of the section.

In the noise-free case we required that the graph $\mathcal G(\Sigma)$
is connected and generic, or, equivalently, $\Sigma$ does not have any 
entry in $\{-1,0,1\}$.
In the presence of noise,  because of problems of identifiability, 
we need a stronger assumptions on the entries of $\Sigma$
corresponding to edges of the graph. In particular, we assume that there exist constants $0<\delta<\gamma<1$ such that
\begin{equation}\label{bound3}
\delta \le  \left| \sigma_{ij}\right| \le \gamma \qquad\mbox{for all }ij\in E~.
\end{equation}
By the product formula \eqref{eq:rhoedge}, this implies that for any two distinct $i,j\in V$,
\[
\left| \sigma_{ij}\right| \ge \delta^D~,
\]
where $D$ is the diameter of the graph.
Although the diameter of $G$ played no role in the noise-free setting,
in the noisy case the recovery guarantees crucially depend on $D$. In
particular, our recovery guarantees are only meaningful when $D$ is
logarithmic in $1/\epsilon$.
To see why this happens, note that under assumption \eqref{bound3},
$|\sigma_{ij}| \le \gamma^{d(i,j)}$ where $d(i,j)$ denotes the distance
of vertex $i$ and vertex $j$ in the tree. This value becomes
indistinguishable from zero by the noisy covariance oracle unless $d(i,j) <
\log(1/\epsilon)/\log(1/\gamma)$.

\begin{remark}
	The assumption that the diameter of $G$  is small is
        relatively mild. Many important real life examples of complex
        networks have small diameter -- these are the so-called small-world networks. For example the diameter of the world wide web, with way over billion nodes \cite{van2016estimating}, is around $19$ \cite{albert1999internet}, while social networks with over six billion individuals are believed to have a diameter of around six \cite{milgram1967small}. Small-world networks have also found applications in brain study \cite{bassett2006small}.
 \end{remark}

Next we show how Algorithm~\ref{tmain} may be modified so that it tolerates noise of
magnitude $\epsilon$ -- in the sense of \eqref{eq:4002}. 

Algorithm~\ref{tmain}  uses the covariance oracle in order to check whether $\det \left(\Sigma _{ij,jk}\right) =0$
for some triples of vertices $i,j,k$, or equivalently, whether $\sigma_{ij}\sigma_{jk}=\sigma_{ik}$.
Also, the algorithm sorts all the correlations $\sigma _{uw}$ for a fixed $w$. 
We show that both of these steps can be correctly executed with a noisy covariance oracle if
\begin{equation}
\label{eq:epscond}
\epsilon \;\le\; \frac{1}{8}\delta^D(1-\gamma^2)~.
\end{equation}
Under \eqref{eq:epscond}, we may choose a value $\tau$ such that
\[
 \tau > 3\epsilon \quad \text{and}\quad    
\tau \leq \delta^D\left(1-\gamma^2\right)-3\epsilon~.
\]
In order to test whether $\sigma_{ij}\sigma_{jk}=\sigma_{ik}$, we use the decision
\begin{eqnarray*}
\text{if}\quad &\left|\wh{\sigma}_{ij}\wh{\sigma}_{jk}-\wh{\sigma}_{ik}\right|< \tau\quad&\text{accept}\\
\text{if}\quad &\left|\wh{\sigma}_{ij}\wh{\sigma}_{jk}-\wh{\sigma}_{ik} \right| > \tau\quad&\text{reject}~.
\end{eqnarray*}

To see why the decision is correct, observe first that for all $i,j\in [n]$,
\begin{equation}
\left|\sigma_{ij}\sigma_{jk}-\wh{\sigma}_{ij}\wh{\sigma}_{jk}\right| 
=\left|\sigma_{jk}\left(\sigma_{ij}-\wh{\sigma}_{ij}\right)+\wh{\sigma}_{ij}\left(\sigma_{jk}-\wh{\sigma}_{jk}\right)\right|\leq 2\epsilon~.\label{eq:305}
\end{equation}
Hence, if $\sigma_{ij}\sigma_{jk}=\sigma_{ik}$, then
\[
   \left|\wh{\sigma}_{ij}\wh{\sigma}_{jk}-\wh{\sigma}_{ik}\right| \le 3\epsilon < \tau
\]
and the decision is correct.

We now treat the case where $\sigma_{ij}\sigma_{jk}\neq \sigma_{ik}$. We consider two cases.

\noindent
\emph{Case I}: There exists a vertex $m$ that separates all $i,j,k$, that is,
\[
\sigma_{im}\sigma_{mk}=\sigma_{ik}\quad \text{and}\quad \sigma_{jm}\sigma_{mk}=\sigma_{jk}~.
\]
In this case
\begin{eqnarray}
\left|\sigma_{ij}\sigma_{jk}-\sigma_{ik} \right|&=& \left|\sigma_{im}\sigma_{jm}^2\sigma_{km}-\sigma_{im}\sigma_{km} \right|\nonumber \\
&=& \left|\sigma_{ik}\right|\left|\sigma_{jm}^2-1\right|\nonumber \\
&\geq&\delta^D\left(1-\gamma^2\right)~.\label{eq:705}
\end{eqnarray}
But then, arguing as above,
\begin{eqnarray*}
\left|\wh{\sigma}_{ij}\wh{\sigma}_{jk}-\wh{\sigma}_{ik}\right| &\geq & \left|\sigma_{ij}\sigma_{jk}-\sigma_{ik} \right|-3\epsilon \\
&\geq & \delta^D\left(1-\gamma^2\right)-3\epsilon > \tau
\end{eqnarray*}
and the decision is once again correct. It remains to consider

\noindent
\emph{Case II}: Either $i$ separates $j$ and $k$ or $k$ separates $i$ and $j$.  Without loss of generality, assume the latter, so that $\sigma_{ij}=\sigma_{ik}\sigma_{jk}$.  Then
\begin{eqnarray*}
\left|\sigma_{ij}\sigma_{jk}-\sigma_{ik}\right| &=& \left|\sigma_{ik}\sigma_{jk}^2-\sigma_{ik}\right|\\
&=& \left|\sigma_{ik}\right|\left|\sigma_{jk}^2-1\right|\\
&\geq & \delta^D \left(1-\gamma^2\right)~.
\end{eqnarray*}
Hence, similarly to Case I, we have
\begin{eqnarray*}
\left|\wh{\sigma}_{ij}\wh{\sigma}_{jk}-\wh{\sigma}_{ik} \right|&\geq & \left|\sigma_{ij}\sigma_{jk}-\sigma_{ik} \right|-3\epsilon \\
&\geq & \delta^D\left(1-\gamma^2\right)-3\epsilon > \tau ~,  
\end{eqnarray*}
which proves correctness of the testing procedure.

The only other ingredient of Algorithm~\ref{tmain} that uses the covariance oracle 
{performs} queries of the form $\left|\sigma _{uw}\right|<\left|\sigma_{vw}\right|$. In the 
presence of a noisy covariance oracle satisfying \eqref{eq:4002}, we may use the following rule:
\begin{eqnarray*}
\text{if }\left|\wh{\sigma}_{uw}\right|<\left|\wh{\sigma}_{vw}\right|-2\epsilon  &\text{accept}\\
\text{if }\left|\wh{\sigma}_{vw}\right|<\left|\wh{\sigma}_{uw}\right|-2\epsilon &\text{reject}\\
\text{otherwise } &\text{do either~.}
\end{eqnarray*}

We now show that the decision rule accepts and rejects correctly. In the first two cases, the decision is 
clearly correct. Hence, we only need to examine the case when
\[
\left|\wh{\sigma}_{uw}\right|\geq\left|\wh{\sigma}_{vw}\right|-2\epsilon \quad \text{and}\quad\left|\wh{\sigma}_{vw}\right|\geq\left|\wh{\sigma}_{uw}\right|-2\epsilon
\]
happen simultaneously. In such case,
\[
\left|\left|\wh{\sigma}_{uw}\right|-\left|\wh{\sigma}_{vw}\right|\right|\leq 2\epsilon ~.
\]
which implies 
$$ \left|\left|\sigma_{uw}\right|-\left|\sigma_{vw}\right|\right|\leq 4\epsilon~.$$
In terms of the ordering of correlations, it is enough that for a central vertex $w$ in Algorithm~\ref{split} the following holds:  if
\begin{enumerate}
	\item [(i)] $w$ does not separate $u$ and $v$ and
	\item [(ii)] $v$ is the neighbor of $w$
\end{enumerate}
then {$|\wh\sigma_{vw}|
>|\wh\sigma_{uw}|$}. Indeed, note that the only thing that matters is that, at each step of the algorithm,  all vertices in the same connected component of $G\setminus w$ are sorted after the unique neighbour of $w$ that belongs in that component. But if (i) and (ii) hold then
\[
\left|\left|\sigma_{uw}\right|-\left|\sigma_{vw}\right|\right|\;=\;|\sigma_{vw}|(1-|\sigma_{uv}|) \geq \delta(1-\gamma)> 4 \epsilon~,
\]
where the last inequality follows by \eqref{eq:epscond}.

This concludes the proof of correctness of the modified procedure under condition \eqref{eq:epscond}.

We close this section by noting that a noisy covariance oracle satisfying \eqref{eq:4002} may easily be
constructed when $\Sigma$ is the covariance matrix of a zero-mean random vector $X=(X_1,\ldots,X_n)$.
All one needs is that, for each pair of indices $i,j\in [n]$ queried by the algorithm, one may 
obtain $N$ i.i.d.\ samples of the pair $(X_i,X_j)$. For example, if all components of $X$ have a bounded
fourth moment, say $\EXP X_i^4 \le \kappa$ for all $i\in [n]$ for some $\kappa >0$, then
$\var(X_iX_j) \le \kappa$, and therefore one may use robust mean estimators (see, e.g., \cite[Theorem 2]{LuMe19})
to estimate $\sigma_{ij} = \EXP[X_iX_j]$
to obtain $\wh{\sigma}_{ij}$ that  satisfy \eqref{eq:epscond} (simultaneously, for all $i,j \in [n]$), with
probability at least $1-\eta$, whenever 
\[
    N \ge 32 \left(\frac{\kappa}{\epsilon}\right)^2 \log \frac{n}{\eta}~.
\]

\acks{We would like to thank Robert Castelo, Vida Dujmovi\'{c}, and David Rossell for helpful discussions. 
GL, VV, and PZ were supported by
the Spanish Ministry of Economy and Competitiveness,
Grant PGC2018-101643-B-I00 and FEDER, EU. 
GL and PZ acknowledge the support of 
``High-dimensional problems in structured probabilistic models - Ayudas Fundaci\'on BBVA a Equipos de Investigaci\'on Cientifica 2017''. 
GL was supported by ``Google Focused Award Algorithms and Learning for AI'' and PZ by Beatriu de Pin\'{o}s grant (BP-2016-00002) and Ram\'{o}n y Cajal (RYC-2017-22544).}

\newpage

\appendix


\appendix

\section{Modified algorithms}\label{app:algs}

Here we present the modified versions of Algorithms \ref{alg:sep0}, \ref{alg:separator}, and \ref{alg:components} that use
conditional covariance information, as described in Section \ref{sec:boundedTW}.

\begin{algorithm}
\label{alg:sep2}
$U \gets \emptyset$\;
$r = {\rm rank}(\Sigma_{AS,BS})$\;
\ForAll{$v\in V$}{\If{${\rm rank}(\Sigma_{ASv,BSv})=r$}{$U\gets U\cup \{v\}$\;}}
$C\gets \{v_0\}$ for some $v_0\in U$\;
\ForAll{$u\in U\setminus \{v_0\}$}{\uIf{${\rm rank}(\Sigma_{ASCu,BSCu})=r$}{
    $C\gets C\cup \{u\}$ \;
  } }
    \Return{ $C$\;}
\caption{${\tt ABSeparator}(V,A,B,S)$}
\end{algorithm}

\begin{algorithm}
Pick a set $W \subset V$ by taking $m$ vertices uniformly at random, where $m$ satisfies (\ref{eq:boundW}) with $r=11k$ and $\delta=1/24$\;
Search exhaustively through all partitions of $W$ into sets
$ A,B$ with $|A|,|B|\leq \frac{2}{3}|W|$, minimizing ${\rm rank}(\Sigma_{AS,BS})$\;
If no balanced split exists, output any partition $ A,B$ of $W$\;
$S'\gets {\tt ABSeparator}(V,A,B,S)$\;
\Return{ $S'$}
\caption{${\tt Separator}(V,S)$}
\end{algorithm}

\begin{algorithm}
  $S'\leftarrow {\tt Separator}(V)$\; 
  $r\gets |S'|+|S|$\;  
  $R \gets \emptyset $\tcp*{will contain one vertex from each $C\in \cC^{S'}$}
   	\For{$v\in V\setminus S'$}{

    $notFound \gets True$\;
   \For{$u\in R$}{
  	\If{${\rm rank}(\Sigma_{uSS',vSS'}) = r+1$ }{$C_u\leftarrow C_u\cup\{v\}$\; $notFound \gets False$\;}}
    \If{$notFound$}{create $C_{v}= \left\{v\right\}$\;$R\leftarrow R\cup\{v\}$\;} }
  	\Return{  $S'$ and all $C_j$\;}
\label{tlsplit22}
\caption{${\tt Components}(V,S)$}
\end{algorithm}

\vskip 0.2in
\bibliography{mref}

\begin{thebibliography}{56}
\providecommand{\natexlab}[1]{#1}
\providecommand{\url}[1]{\texttt{#1}}
\expandafter\ifx\csname urlstyle\endcsname\relax
  \providecommand{\doi}[1]{doi: #1}\else
  \providecommand{\doi}{doi: \begingroup \urlstyle{rm}\Url}\fi

\bibitem[Abu-Ata and Dragan(2016)]{abu2016metric}
Muad Abu-Ata and Feodor~F Dragan.
\newblock Metric tree-like structures in real-world networks: an empirical
  study.
\newblock \emph{Networks}, 67\penalty0 (1):\penalty0 49--68, 2016.

\bibitem[Adcock et~al.(2013)Adcock, Sullivan, and Mahoney]{adcock2013tree}
Aaron~B Adcock, Blair~D Sullivan, and Michael~W Mahoney.
\newblock Tree-like structure in large social and information networks.
\newblock In \emph{2013 IEEE 13th International Conference on Data Mining},
  pages 1--10. IEEE, 2013.

\bibitem[Albert et~al.(1999)Albert, Jeong, and
  Barab{\'a}si]{albert1999internet}
R{\'e}ka Albert, Hawoong Jeong, and Albert-L{\'a}szl{\'o} Barab{\'a}si.
\newblock Internet: Diameter of the world-wide web.
\newblock \emph{nature}, 401\penalty0 (6749):\penalty0 130, 1999.

\bibitem[Arnborg et~al.(1987)Arnborg, Corneil, and
  Proskurowski]{arnborg1987complexity}
Stefan Arnborg, Derek~G Corneil, and Andrzej Proskurowski.
\newblock Complexity of finding embeddings in ak-tree.
\newblock \emph{SIAM Journal on Algebraic Discrete Methods}, 8\penalty0
  (2):\penalty0 277--284, 1987.

\bibitem[Banerjee et~al.(2008)Banerjee, Ghaoui, and
  d'Aspremont]{banerjee2008model}
Onureena Banerjee, Laurent~El Ghaoui, and Alexandre d'Aspremont.
\newblock Model selection through sparse maximum likelihood estimation for
  multivariate gaussian or binary data.
\newblock \emph{Journal of Machine Learning Research}, 9\penalty0
  (Mar):\penalty0 485--516, 2008.

\bibitem[Bassett and Bullmore(2006)]{bassett2006small}
Danielle~Smith Bassett and ED~Bullmore.
\newblock Small-world brain networks.
\newblock \emph{The neuroscientist}, 12\penalty0 (6):\penalty0 512--523, 2006.

\bibitem[Bello and Honorio(2018)]{NEURIPS2018_a0b45d1b}
Kevin Bello and Jean Honorio.
\newblock Computationally and statistically efficient learning of causal bayes
  nets using path queries.
\newblock In S.~Bengio, H.~Wallach, H.~Larochelle, K.~Grauman, N.~Cesa-Bianchi,
  and R.~Garnett, editors, \emph{Advances in Neural Information Processing
  Systems}, volume~31. Curran Associates, Inc., 2018.
\newblock URL
  \url{https://proceedings.neurips.cc/paper/2018/file/a0b45d1bb84fe1bedbb8449764c4d5d5-Paper.pdf}.

\bibitem[Bodlaender(1998)]{bodlaender1998partial}
Hans~L Bodlaender.
\newblock A partial $k$-arboretum of graphs with bounded treewidth.
\newblock \emph{Theoretical Computer Science}, 209\penalty0 (1-2):\penalty0
  1--45, 1998.

\bibitem[Bodlaender et~al.(2016)Bodlaender, Drange, Dregi, Fomin, Lokshtanov,
  and Pilipczuk]{bodlaender2016c}
Hans~L Bodlaender, Pål~Grǿnås Drange, Markus~S Dregi, Fedor~V Fomin, Daniel
  Lokshtanov, and Micha{\l} Pilipczuk.
\newblock A $c^kn$ 5-approximation algorithm for treewidth.
\newblock \emph{SIAM Journal on Computing}, 45\penalty0 (2):\penalty0 317--378,
  2016.

\bibitem[Brown and Truszkowski(2011)]{brown2011fast}
Daniel~G Brown and Jakub Truszkowski.
\newblock Fast error-tolerant quartet phylogeny algorithms.
\newblock In \emph{Annual Symposium on Combinatorial Pattern Matching}, pages
  147--161. Springer, 2011.

\bibitem[Brown and Truszkowski(2012)]{brown2012fast}
Daniel~G Brown and Jakub Truszkowski.
\newblock Fast phylogenetic tree reconstruction using locality-sensitive
  hashing.
\newblock In \emph{International Workshop on Algorithms in Bioinformatics},
  pages 14--29. Springer, 2012.

\bibitem[Cai et~al.(2016)Cai, Liu, Zhou, et~al.]{cai2016estimating}
T~Tony Cai, Weidong Liu, Harrison~H Zhou, et~al.
\newblock Estimating sparse precision matrix: Optimal rates of convergence and
  adaptive estimation.
\newblock \emph{The Annals of Statistics}, 44\penalty0 (2):\penalty0 455--488,
  2016.

\bibitem[Chan et~al.(2016)Chan, Stumpf, and Babtie]{Chan082099}
Thalia~E Chan, Michael~PH Stumpf, and Ann~C Babtie.
\newblock Network inference and hypotheses-generation from single-cell
  transcriptomic data using multivariate information measures.
\newblock \emph{bioRxiv}, 2016.
\newblock \doi{10.1101/082099}.
\newblock URL \url{https://www.biorxiv.org/content/early/2016/10/20/082099}.

\bibitem[Chandrasekaran and Jordan(2013)]{chandrasekaran2013computational}
Venkat Chandrasekaran and Michael~I Jordan.
\newblock Computational and statistical tradeoffs via convex relaxation.
\newblock \emph{Proceedings of the National Academy of Sciences}, 110\penalty0
  (13):\penalty0 E1181--E1190, 2013.

\bibitem[Chandrasekaran et~al.(2012)Chandrasekaran, Srebro, and
  Harsha]{chandrasekaran2012complexity}
Venkat Chandrasekaran, Nathan Srebro, and Prahladh Harsha.
\newblock Complexity of inference in graphical models.
\newblock \emph{arXiv preprint arXiv:1206.3240}, 2012.

\bibitem[Chow and Liu(1968)]{chow1968approximating}
C~Chow and Cong Liu.
\newblock Approximating discrete probability distributions with dependence
  trees.
\newblock \emph{IEEE Transactions on Information Theory}, 14\penalty0
  (3):\penalty0 462--467, 1968.

\bibitem[Dasarathy et~al.(2016)Dasarathy, Singh, Balcan, and
  Park]{dasarathy2016active}
Gautamd Dasarathy, Aarti Singh, Maria-Florina Balcan, and Jong~H Park.
\newblock Active learning algorithms for graphical model selection.
\newblock In \emph{Artificial Intelligence and Statistics}, pages 1356--1364,
  2016.

\bibitem[Devroye and Lugosi(2000)]{DeLu00}
L.~Devroye and G.~Lugosi.
\newblock \emph{Combinatorial Methods in Density Estimation}.
\newblock Springer-Verlag, New York, 2000.

\bibitem[Drton and Maathuis(2017)]{drton2017structure}
Mathias Drton and Marloes~H Maathuis.
\newblock Structure learning in graphical modeling.
\newblock \emph{Annual Review of Statistics and Its Application}, 4:\penalty0
  365--393, 2017.

\bibitem[Edwards et~al.(2010)Edwards, De~Abreu, and
  Labouriau]{edwards2010selecting}
David Edwards, Gabriel~CG De~Abreu, and Rodrigo Labouriau.
\newblock Selecting high-dimensional mixed graphical models using minimal aic
  or bic forests.
\newblock \emph{BMC Bioinformatics}, 11\penalty0 (1):\penalty0 18, 2010.

\bibitem[Feige and Mahdian(2006)]{feige2006finding}
Uriel Feige and Mohammad Mahdian.
\newblock Finding small balanced separators.
\newblock In \emph{Proceedings of the Thirty-eighth Annual ACM Symposium on
  Theory of Computing}, pages 375--384. ACM, 2006.

\bibitem[Frydenberg(1990)]{frydenberg1990marginalization}
Morten Frydenberg.
\newblock Marginalization and collapsibility in graphical interaction models.
\newblock \emph{The Annals of Statistics}, pages 790--805, 1990.

\bibitem[Hajiaghayi and Hajiaghayi(2003)]{hajiaghayi2003note}
Mohammad~Taghi Hajiaghayi and Mahdi Hajiaghayi.
\newblock A note on the bounded fragmentation property and its applications in
  network reliability.
\newblock \emph{European Journal of Combinatorics}, 24\penalty0 (7):\penalty0
  891--896, 2003.

\bibitem[Hajiaghayi and Nishimura(2007)]{hajiaghayi2007subgraph}
MohammadTaghi Hajiaghayi and Naomi Nishimura.
\newblock Subgraph isomorphism, log-bounded fragmentation, and graphs of
  (locally) bounded treewidth.
\newblock \emph{Journal of Computer and System Sciences}, 73\penalty0
  (5):\penalty0 755--768, 2007.

\bibitem[Harary(1969)]{harary6graph}
Frank Harary.
\newblock \emph{Graph theory}.
\newblock Addison-Wesley, Reading, MA, 1969.

\bibitem[H{\o}jsgaard et~al.(2012)H{\o}jsgaard, Edwards, and
  Lauritzen]{hojsgaard2012graphical}
S{\o}ren H{\o}jsgaard, David Edwards, and Steffen Lauritzen.
\newblock \emph{Graphical models with R}.
\newblock Springer Science \& Business Media, 2012.

\bibitem[Hsieh et~al.(2013)Hsieh, Sustik, Dhillon, Ravikumar, and
  Poldrack]{hsieh2013big}
Cho-Jui Hsieh, M{\'a}ty{\'a}s~A Sustik, Inderjit~S Dhillon, Pradeep~K
  Ravikumar, and Russell Poldrack.
\newblock Big \& quic: Sparse inverse covariance estimation for a million
  variables.
\newblock In \emph{Advances in neural information processing systems}, pages
  3165--3173, 2013.

\bibitem[Huang et~al.(2010)Huang, Li, Sun, Ye, Fleisher, Wu, Chen, Reiman,
  Initiative, et~al.]{huang2010learning}
Shuai Huang, Jing Li, Liang Sun, Jieping Ye, Adam Fleisher, Teresa Wu, Kewei
  Chen, Eric Reiman, Alzheimer's Disease~NeuroImaging Initiative, et~al.
\newblock Learning brain connectivity of alzheimer's disease by sparse inverse
  covariance estimation.
\newblock \emph{NeuroImage}, 50\penalty0 (3):\penalty0 935--949, 2010.

\bibitem[Hwang et~al.(2018)Hwang, Lee, and Bang]{hwang}
Byungjin Hwang, Ji~Hyun Lee, and Duhee Bang.
\newblock Single-cell rna sequencing technologies and bioinformatics pipelines.
\newblock \emph{Experimental \& Molecular Medicine}, 50\penalty0 (8):\penalty0
  96, 2018.
\newblock \doi{10.1038/s12276-018-0071-8}.
\newblock URL \url{https://doi.org/10.1038/s12276-018-0071-8}.

\bibitem[Jagadish and Sen(2013)]{10.1007/978-3-642-40935-6_14}
M.~Jagadish and Anindya Sen.
\newblock Learning a bounded-degree tree using separator queries.
\newblock In Sanjay Jain, R{\'e}mi Munos, Frank Stephan, and Thomas Zeugmann,
  editors, \emph{Algorithmic Learning Theory}, pages 188--202, Berlin,
  Heidelberg, 2013. Springer Berlin Heidelberg.

\bibitem[Karger and Srebro(2001)]{karger2001learning}
David Karger and Nathan Srebro.
\newblock Learning markov networks: Maximum bounded tree-width graphs.
\newblock In \emph{Proceedings of the Twelfth Annual ACM-SIAM Symposium on
  Discrete Algorithms}, pages 392--401. Society for Industrial and Applied
  Mathematics, 2001.

\bibitem[Kwisthout et~al.(2010)Kwisthout, Bodlaender, and van~der
  Gaag]{kwisthout2010necessity}
Johan Kwisthout, Hans~L Bodlaender, and Linda~C van~der Gaag.
\newblock The necessity of bounded treewidth for efficient inference in
  bayesian networks.
\newblock In \emph{ECAI}, volume 215, pages 237--242, 2010.

\bibitem[Lauritzen(1996)]{lauritzen1996graphical}
Steffen~L Lauritzen.
\newblock \emph{Graphical models}.
\newblock Clarendon Press, 1996.

\bibitem[Liu et~al.(2009)Liu, Lafferty, and Wasserman]{liu2009nonparanormal}
Han Liu, John Lafferty, and Larry Wasserman.
\newblock The nonparanormal: Semiparametric estimation of high dimensional
  undirected graphs.
\newblock \emph{Journal of Machine Learning Research}, 10\penalty0
  (Oct):\penalty0 2295--2328, 2009.

\bibitem[Liu et~al.(2012)Liu, Han, Yuan, Lafferty, Wasserman,
  et~al.]{liu2012high}
Han Liu, Fang Han, Ming Yuan, John Lafferty, Larry Wasserman, et~al.
\newblock High-dimensional semiparametric gaussian copula graphical models.
\newblock \emph{The Annals of Statistics}, 40\penalty0 (4):\penalty0
  2293--2326, 2012.

\bibitem[Loh and Wainwright(2013)]{lohright}
Po-Ling Loh and Martin~J. Wainwright.
\newblock Structure estimation for discrete graphical models: generalized
  covariance matrices and their inverses.
\newblock \emph{Ann. Statist.}, 41\penalty0 (6):\penalty0 3022--3049, 2013.
\newblock ISSN 0090-5364.
\newblock \doi{10.1214/13-AOS1162}.
\newblock URL \url{https://doi.org/10.1214/13-AOS1162}.

\bibitem[Lugosi and Mendelson(2019)]{LuMe19}
G{\'a}bor Lugosi and Shahar Mendelson.
\newblock Mean estimation and regression under heavy-tailed distributions: A
  survey.
\newblock \emph{Foundations of Computational Mathematics}, 19\penalty0
  (5):\penalty0 1145--1190, 2019.

\bibitem[Maniu et~al.(2019)Maniu, Senellart, and Jog]{maniu2019experimental}
Silviu Maniu, Pierre Senellart, and Suraj Jog.
\newblock An experimental study of the treewidth of real-world graph data
  (extended version).
\newblock \emph{arXiv preprint arXiv:1901.06862}, 2019.

\bibitem[Meinshausen and B{\"u}hlmann(2010)]{meinshausen2010stability}
Nicolai Meinshausen and Peter B{\"u}hlmann.
\newblock Stability selection.
\newblock \emph{Journal of the Royal Statistical Society: Series B (Statistical
  Methodology)}, 72\penalty0 (4):\penalty0 417--473, 2010.

\bibitem[Milgram(1967)]{milgram1967small}
Stanley Milgram.
\newblock The small world problem.
\newblock \emph{Psychology today}, 2\penalty0 (1):\penalty0 60--67, 1967.

\bibitem[Panagiotou and Steger(2010)]{panagiotou2010maximal}
Konstantinos Panagiotou and Angelika Steger.
\newblock Maximal biconnected subgraphs of random planar graphs.
\newblock \emph{ACM Transactions on Algorithms (TALG)}, 6\penalty0
  (2):\penalty0 31, 2010.

\bibitem[Pearl(2009)]{MR2548166}
Judea Pearl.
\newblock \emph{Causality}.
\newblock Cambridge University Press, Cambridge, second edition, 2009.
\newblock ISBN 978-0-521-89560-6; 0-521-77362-8.
\newblock \doi{10.1017/CBO9780511803161}.
\newblock URL \url{https://doi.org/10.1017/CBO9780511803161}.
\newblock Models, reasoning, and inference.

\bibitem[Price et~al.(2010)Price, Dehal, and Arkin]{price2010fasttree}
Morgan~N Price, Paramvir~S Dehal, and Adam~P Arkin.
\newblock Fasttree 2--approximately maximum-likelihood trees for large
  alignments.
\newblock \emph{PloS one}, 5\penalty0 (3):\penalty0 e9490, 2010.

\bibitem[Rossell and Zwiernik(2020)]{rossell2020dependence}
David Rossell and Piotr Zwiernik.
\newblock Dependence in elliptical partial correlation graphs.
\newblock \emph{arXiv preprint arXiv:2004.13779}, 2020.

\bibitem[Rudi et~al.(2015)Rudi, Camoriano, and Rosasco]{rudi2015less}
Alessandro Rudi, Raffaello Camoriano, and Lorenzo Rosasco.
\newblock Less is more: Nystr{\"o}m computational regularization.
\newblock In \emph{Advances in Neural Information Processing Systems}, pages
  1657--1665, 2015.

\bibitem[Scheines et~al.(1998)Scheines, Spirtes, Glymour, Meek, and
  Richardson]{scheines1998tetrad}
Richard Scheines, Peter Spirtes, Clark Glymour, Christopher Meek, and Thomas
  Richardson.
\newblock The tetrad project: Constraint based aids to causal model
  specification.
\newblock \emph{Multivariate Behavioral Research}, 33\penalty0 (1):\penalty0
  65--117, 1998.

\bibitem[Spirtes et~al.(2000)Spirtes, Glymour, and Scheines]{MR1815675}
Peter Spirtes, Clark Glymour, and Richard Scheines.
\newblock \emph{Causation, prediction, and search}.
\newblock Adaptive Computation and Machine Learning. MIT Press, Cambridge, MA,
  second edition, 2000.
\newblock ISBN 0-262-19440-6.
\newblock With additional material by David Heckerman, Christopher Meek,
  Gregory F. Cooper and Thomas Richardson, A Bradford Book.

\bibitem[Sullivant et~al.(2010)Sullivant, Talaska, and
  Draisma]{sullivant2010trek}
Seth Sullivant, Kelli Talaska, and Jan Draisma.
\newblock Trek separation for {G}aussian graphical models.
\newblock \emph{The Annals of Statistics}, pages 1665--1685, 2010.

\bibitem[Van~den Bosch et~al.(2016)Van~den Bosch, Bogers, and
  De~Kunder]{van2016estimating}
Antal Van~den Bosch, Toine Bogers, and Maurice De~Kunder.
\newblock Estimating search engine index size variability: a 9-year
  longitudinal study.
\newblock \emph{Scientometrics}, 107\penalty0 (2):\penalty0 839--856, 2016.

\bibitem[Vapnik and Chervonenkis(1971)]{VaCh71}
V.N. Vapnik and A.Ya. Chervonenkis.
\newblock On the uniform convergence of relative frequencies of events to their
  probabilities.
\newblock \emph{Theory of Probability and its Applications}, 16:\penalty0
  264--280, 1971.

\bibitem[Wainwright et~al.(2008)Wainwright, Jordan,
  et~al.]{wainwright2008graphical}
Martin~J Wainwright, Michael~I Jordan, et~al.
\newblock Graphical models, exponential families, and variational inference.
\newblock \emph{Foundations and Trends in Machine Learning}, 1\penalty0
  (1--2):\penalty0 1--305, 2008.

\bibitem[Yuan and Lin(2007)]{yuan2007model}
Ming Yuan and Yi~Lin.
\newblock Model selection and estimation in the gaussian graphical model.
\newblock \emph{Biometrika}, 94\penalty0 (1):\penalty0 19--35, 2007.

\bibitem[Zhang(2005)]{zhang}
Fuzhen Zhang, editor.
\newblock \emph{The {S}chur complement and its applications}, volume~4 of
  \emph{Numerical Methods and Algorithms}.
\newblock Springer-Verlag, New York, 2005.
\newblock ISBN 0-387-24271-6.
\newblock \doi{10.1007/b105056}.
\newblock URL \url{https://doi.org/10.1007/b105056}.

\bibitem[Zhang et~al.(2018)Zhang, Fattahi, and Sojoudi]{zhang2018large}
Richard~Y Zhang, Salar Fattahi, and Somayeh Sojoudi.
\newblock Large-scale sparse inverse covariance estimation via thresholding and
  max-det matrix completion.
\newblock \emph{arXiv preprint arXiv:1802.04911}, 2018.

\bibitem[Zhang et~al.(2011)Zhang, Zhao, He, Lu, Cao, Liu, Hao, Liu, and
  Chen]{zhang2011inferring}
Xiujun Zhang, Xing-Ming Zhao, Kun He, Le~Lu, Yongwei Cao, Jingdong Liu, Jin-Kao
  Hao, Zhi-Ping Liu, and Luonan Chen.
\newblock Inferring gene regulatory networks from gene expression data by path
  consistency algorithm based on conditional mutual information.
\newblock \emph{Bioinformatics}, 28\penalty0 (1):\penalty0 98--104, 2011.

\bibitem[Zwiernik(2019)]{ltm}
Piotr Zwiernik.
\newblock Latent tree models.
\newblock In \emph{Handbook of graphical models}, Chapman \& Hall/CRC Handb.
  Mod. Stat. Methods, pages 265--288. CRC Press, Boca Raton, FL, 2019.

\end{thebibliography}

\end{document}